\documentclass[smallextended,referee,envcountsect,]{svjour3}
\usepackage{graphicx} 
\usepackage{amsmath}
\usepackage{amssymb}
\usepackage{authblk}
\usepackage{tikz, pgf}
\usepackage{xcolor} 
\usepackage[displaymath,mathlines]{lineno}
\usetikzlibrary{decorations.markings}
\usepackage[english]{babel}
\usepackage[T1]{fontenc}
\usepackage[utf8]{inputenc}
\usepackage{amsfonts}
\usepackage{mathrsfs}
\usepackage{xfrac}
\usepackage{lmodern}
\DeclareFontFamily{OMX}{lmex}{}
\DeclareFontShape{OMX}{lmex}{m}{n}{<-> lmex10}{}
\usepackage{fix-cm}
\usepackage{listings}
\usepackage{booktabs} 
\usepackage{enumerate}   
\usetikzlibrary{arrows}
\usepackage{float}
\usepackage{subcaption}
\usepackage{hyperref}
\usepackage{bbm}

\allowdisplaybreaks
\newcommand{\R}{{\mathbb{R}}}

\newcommand{\lra}{\longrightarrow}

\newcommand{\Wass}{\mathbf{d}}

\newcommand{\G}{\Gamma}

\newcommand{\cA}{{\mathcal A}}

\newcommand{\cC}{{\mathcal C}}

\newcommand{\cM}{{\mathcal M}}

\newcommand{\cV}{{\mathcal V}}
\newcommand{\cT}{{\mathcal T}}

\renewcommand{\phi}{\varphi}
\renewcommand{\a}{{\alpha}}
\renewcommand{\b}{{\beta}}
\newcommand{\g}{{\gamma}}
\renewcommand{\d}{{\delta}}

\renewcommand{\th}{\theta}

\newcommand{\bm}{{\bar m}}

\def\squareforqed{\hbox{\rlap{$\sqcap$}$\sqcup$}}
\newcommand{\supp}{\mathrm{supp}}

\usepackage{fancyhdr}

\newtheorem{theo}{\bf Theorem}[section]

\newtheorem{prop}[theo]{\bf Proposition}

\providecommand{\keywords}[1]{\textbf{Keywords.} #1}
\providecommand{\mathsubjclass}[2]{\textbf{2020 Mathematics Subject Classification.} #1}

\begin{document}

\title{A Network Model for Urban Planning}

\author{Fabio Camilli \and Adriano Festa \and Luciano Marzufero.}

\institute{Fabio Camilli \at
               Dipartimento di Ingegneria e Geologia, Univ. ``G. D'Annunzio'' Chieti-Pescara,\\
		viale Pindaro 42, 65127 Pescara (Italy), \\
               fabio.camilli@unich.it
           \and
               Adriano Festa  \at
               Dipartimento di Scienze Matematiche ``G. L. Lagrange'', Politecnico di Torino \\
               Corso Duca degli Abruzzi 24, 10129 Torino, Italy\\
               adriano.festa@polito.it
          \and
	      Luciano Marzufero \at
	      Faculty of Economics and Management, Free University of Bozen-Bolzano\\		
	      Piazza Universit\`a 1, 39100 Bolzano, Italy\\
	      luciano.marzufero@unibz.it
}

\date{Received: date / Accepted: date}

\maketitle

\begin{abstract}
We study a mathematical model to describe the evolution of a city, which is determined by the interaction of two large populations of agents,  workers and firms. The map of the city is described  by a network with the edges representing at the same time residential areas and communication routes. The two populations  compete for space while interacting through the labour market. The resulting model is described by a two population Mean-Field Game system coupled with an Optimal Transport problem. We prove existence and uniqueness of the solution and we provide several numerical simulations.
\end{abstract}
\par\smallskip\noindent
\mathsubjclass{91A16; 49Q22; 35R02; 49N80.}

\par\smallskip\noindent
Communicated by Fausto Gozzi.

\section{Introduction}
Currently, 55\% of the global population lives in urban areas, a number expected to rise to 68\% by 2050, according to the United Nations. The growth is driven by population increases and rural-to-urban migration for economic opportunities.  This shift has emerged as a focal point in economic literature. Mathematical frameworks aimed at analyzing and comprehending this phenomenon have been proposed  in \cite{acpt,bagagiolo2,bagagiolo1,buttazzo,carlier,lucas}.
\par
A model for urban planning has been recently introduced in \cite{barilla}, where the shape of the city is  determined by the interaction  between two different populations representing workers and, respectively, firms, each one formed by a large number of indistinguishable agents. They compete for land use, paying a rent for space occupation. Moreover they interact through the labour market since wages are paid by firms to workers, who choose the residence and workplace so as to maximize the revenue, i.e. wage minus commuting cost. Instead the strategy of the firms aims to minimize the labour cost. It follows that  agents of each population solve  a stochastic control problem  where the cost functional depends on the distribution of similar agents and the interaction with the ones of the other population. Moreover, an additional condition expressing equilibrium of the labour market is imposed. From a mathematical point of view, the previous model leads to a two population Mean-field Game  (MFG in short) system \cite{lasry2007mean,huang2006large,cardaliaguet2012notes} coupled with an Optimal Transport problem \cite{villani2,santambrogio} between the distributions of the two populations.
\par
While the model in \cite{barilla} has been studied in the periodic Euclidean case, in this paper we consider a similar problem, but in the case that the state space is given by a network. Indeed, such a geometric structure can be interpreted as the plan of a city with the edges representing both residential areas and communication routes. Each area has specific characteristics that influence the rental cost, as available service, density of population, pollution, etc. Furthermore, the speed of the connections, which can be described by the length of the edges and other parameters in the problem, plays an important role since it affects the commuting cost between residential areas and work places. Hence the study of urban planning problem on networks introduces some interesting peculiarities with respect to the Euclidean case. Concerning the mathematical approach, in \cite{barilla} existence of a solution is obtained via a variational technique which requires a symmetric interaction among the two populations: the land rent is the same for workers and firms and depends only on the total density of the two populations. Here we prove existence by means of a  fixed-point argument and this approach does not require a symmetric behavior between the two populations. Indeed, it seems to be natural to assume that people prefer to avoid to live near polluting factories or in overcrowded residential areas, while industries tend to cluster to take advantage of a more effective transport system. Moreover, with respect to \cite{barilla}, we also prove an uniqueness result under a monotonicity assumption involving the cost functions of the two populations.\par
 Classical theories of urban land use and spatial equilibrium developed in the works \cite{Alonso1964,Mills1967,Muth1969}  describe how agents' location choices result from the trade-off between commuting costs and land rents. More recent contributions, such as \cite{Fujita1999,Duranton2004}, have highlighted how agglomeration forces and transport infrastructure shape urban structures and labor markets. In this context, the present framework captures the influence of commuting costs, land-use competition, and agglomeration effects through population densities and congestion-dependent rents. At the same time, it deliberately abstracts from features that are central in many classical urban economics models, such as discrete location choices, zoning constraints, heterogeneous firms or workers, and explicit housing market institutions.
\par
The paper is organized as follows. In Section \ref{sec:model}, we introduce the model and provide its mathematical formulation, which consists of an MFG system coupled with an Optimal Transport problem. In Section \ref{sec:existence}, we investigate the existence and uniqueness of solutions to the problem. Section \ref{sec:discretization} focuses on the numerical approximation, where we describe various examples to demonstrate the model's sensitivity to the network structure and model parameters. Finally, in Appendix \ref{app:A}, we present and prove several useful results related to the Hamilton-Jacobi equation, the Fokker-Planck equation and Optimal Transport on networks. Many of these results are well-known in the Euclidean case and we discuss the adaptation to the network framework.

\section{The Two Population Model on Network}\label{sec:model}
In this section, we describe the mathematical model and derive the MFG system coupled with the Optimal Transport problem which characterizes the equilibria of the problem.\\
The state space of the problem is given by a network $\Gamma\subset\R^d$ which is composed by a finite collection of bounded edges $\mathcal{E}:=\left\{ \Gamma_{\alpha}:\alpha\in\cA\right\}$ connecting a finite collection of vertices $\cV:=\left\{\nu_{j}:j\in J\right\}$. Two edges can intersect only at the vertices, i.e. for $\alpha,\beta\in\cA$ with $\alpha\not=\beta$, then $\Gamma_\alpha\cap \Gamma_\beta$ is either empty or made of a single vertex. For simplicity, we assume that the edges are segments and, for $\Gamma_{\alpha}\in\mathcal{E}$ connecting two vertices $\nu_i$ and $\nu_j$ with $i<j$, we consider the parametrization $\pi_\alpha: [0,\ell_\alpha]\lra \Gamma_\alpha$ given by
\begin{equation}\label{eq:par_arc}	 
\pi_\alpha(y)=[y\nu_j+(\ell_\alpha-y)\nu_i]\ell_\alpha^{-1}\qquad\text{for } y\in[0,\ell_\alpha],
\end{equation} 
where $\ell_\alpha$ is the length of the edge. We   denote by $\cA_{i}=\left\{ \alpha\in\cA:\nu_{i}\in\Gamma_{\alpha}\right\}$  the set of indices of edges that are adjacent to the vertex $\nu_{i}$. For a function $v:\G\lra\R$,   $v_\a$  is the   restriction of $v$ to $\G_\alpha$, i.e. $v_{\alpha} (y):=v|_{\Gamma_{\alpha}}\circ\pi_{\alpha} (y)$ for $y\in (0,\ell_\alpha)$. 
Derivatives   of a function on the network are defined in standard way with outward derivatives at the vertices  (see \cite{achdou} for example). We set $|\Gamma|=\sum_{a\in \cA}\ell_\alpha$, where $\ell_\alpha$ as in \eqref{eq:par_arc} and for $v:\G\to\R$ we define
\begin{equation}\label{eq:int}
	\int_\Gamma v(x)dx=\frac{1}{|\Gamma|}\sum_{\alpha\in\cA} \int_{\Gamma_\alpha} v_\alpha(x)dx.
\end{equation}
\par\noindent\textbf{The Optimal Control Problems for the Populations.}
We consider an evolutive model where the distribution of workers and firms at time $t\in [0,T]$ is described by probability distribution $m_1(t)$ and, respectively, $m_2(t)$. The initial distributions $m^1_0$ and $m^2_0$ are given probability measures on the network $\G$.\\
The representative agent of the workers population solves an optimal control problem with dynamics given by a network Markov process $(X^1(s), \alpha^1(s))$ such that $X^1(s)\in \G_{\a^1(s)}$ with $X^1(t)=x\in \G$ (see, for example, \cite{freidlin2} for stochastic processes on networks). Inside the edge $\Gamma_{\alpha^1(s)}$, the process is characterized by the stochastic differential equation
$$
dX^1(s)= u^1_{\a^1(s)}(X^1(s))ds+\sqrt{2\mu^1_{\alpha^1(s)}}dB^1(s),
$$
where, for $\alpha\in \cA$, $u^1_{\a}:\G_\a\lra\R$ is a feedback control law with value in the compact set $U^1_{\a}=[\underline{u^1_{\a}}, \overline{u^1_{\a}}]$, $B^1(s)$  a one dimensional Wiener process and $\mu^1_\a>0$. Note that, since $\mu^1_\a>0$ for any $\a\in\cA$, $\mathbb{P}(X(t)\in \cV: t\in [0,T])=0$ (see e.g. \cite{ohavi}). When the agent arrives at a vertex $\nu_i$, enters one of the adjacent edges $\G_\a \in \cA_{i}$ with probability
$$
	p^1_{j\a}=\frac{\gamma^1_{j\alpha}\mu^1_{\alpha}}{\sum_{\a\in\cA_j} \gamma^1_{j\alpha}\mu^1_{\alpha}},
$$
where $\gamma^1_{j\alpha}>0$, $j\in J$ and $\a\in\cA$, are parameters of the model which may represent a preference for one connection over another. The worker living at $x\in \Gamma$ at time $t$ {\it minimizes} the cost functional
$$
\mathbb{E}_{x,t}\int_t^T \Big[ L^1(u^1_{\alpha^1(s)},X^1(s))-r(s,X^1(s))+R^1[m^1(s),m^2(s)](X^1(s))\Big]ds,
$$
where $L^1$ represents the mobility cost, $r$ the revenue that individuals, living at location $x$, bring home (net of commuting cost) and $R^1$ the rent cost.\\
The optimal control problem for the  firm agent is similar. The dynamics is given by a network Markov process  $(X^2(s), \a^2(s))$ with $X^2(s)\in \G_{\a^2(s)}$ and $X^2(t)=x\in\G$, which, inside the edge, is described by the stochastic differential equation 
$$
dX^2(s)=u^2_{\a^2(s)}(X^2(s))ds+\sqrt{2\mu^2_{\alpha^2(s)}}dB^2(s),
$$
where, for $\alpha\in \cA$, $u^2_{\a}:\G_\a\lra\R$ is a feedback control law with value in the compact set $U^2_{\a}=[\underline{u^2_{\a}}, \overline{u^2_{\a}}]$, $B^2(s)$ a one dimensional Wiener process and $\mu^2_\a>0$.  As before, the agent spends $0$-time at the vertices with probability one and enters in one of the adjacent edge with probability
\begin{equation}\label{coeff2}
	p^2_{j\a}=\frac{\gamma^2_{j\alpha}\mu^2_{\alpha}}{\sum_{\a\in\cA_j} \gamma^2_{j\alpha}\mu^2_{\alpha}} ,
\end{equation}
where $\gamma^2_{j\alpha}>0$, $j\in J$ and $\a\in\cA$. The cost functional to be minimized is given by
$$
\mathbb{E}_{x,t}\int_t^T \Big[ L^2(u^2_{\alpha^2(s)}, X^2(s))+w(s,X^2(s))+R^2[m^1(s),m^2(s)](X^2(s))\Big]ds,
$$
where $L^2$ is the mobility cost for the firm, $w$ the wage that the firm, located at $x$, pays to the workers and $R^2$ the rent cost.
\begin{remark}
In \cite{barilla}, it is assumed that the rent cost is the same for both the populations and depends only on the total demand,  i.e. $R^1[m_1(t),m_2(t)]=R^2[m_1(t),m_2(t)]=f(m_1(t)+m_2(t))$ for some increasing function $f:\R^+\lra\R^+$. Here we  consider different and more general coupling costs which also take into account different needs for the two populations, see Section \ref{sec:discretization} for details.
\end{remark}
\par\noindent\textbf{The Equilibrium Condition for the Labour Market.}
In addition to the used space, workers and companies also interact through the revenue function $r$, the wage function $w$ and the free mobility of the labour market. Indeed, denoted with $c(x,y)$ the cost of commuting, at time $t\in(0, T)$ people living at $x$ choose to work at location $y$ which maximizes their revenue, i.e.
\begin{equation}\label{revenue}
r(t,x)=\max_{y\in\G}\{w(t,y)-c(x,y)\}.
\end{equation}
In the same way, firms located at $y$ at time $t$ hire workers to minimize the wage, i.e.
\begin{equation}\label{wage}
w(t,y)=\min_{x\in\G}\{r(t,x)+c(x,y)\}.
\end{equation}
An equilibrium in the labour market is a configuration where there is no incentive for workers to change the living place and for firms to move in another place. This condition can be expressed   in the following way (see \cite{barilla}): the couple of continuous functions $(w(t,\cdot), r(t,\cdot))$ induces an equilibrium in the labour market at time $t\in (0,T)$ if there is a transport plan $\gamma$ between $m_1(t)$ and $m_2(t)$, i.e. $\gamma$ has marginals $m_1(t)$ and $m_2(t)$, such that  
\begin{equation}\label{equilibrium}
	w(t, y)-r(t, x)=c(x, y) \quad\text{on $\supp(\gamma)$},
\end{equation}
where $\supp(\gamma)$ is the support of $\gamma$. Hence  the equilibrium condition is equivalent to find an    optimal  transport  plan for  the   problem 
\begin{equation}\label{OT_primal}
	 \cC(m_1(t), m_2(t))=\inf_{\gamma\in\Pi(m_1, m_2)}\int_{\Gamma\times\Gamma}c(x, y)d\gamma(x, y),
\end{equation}
where $\Pi$ are the transport plans between $m_1$ and $m_2$.
Taking into account \eqref{revenue}, \eqref{wage},\eqref{equilibrium} and  the Kantorovich duality (see e.g. \cite{santambrogio,mugnolo,toledo}), $ \cC(m^1(t), m^2(t))$ can be equivalently rewritten in the dual form  as
\begin{align}
\label{dualformulationmodel}
\cC(m^1(t), m^2(t)) 
	 =\sup\Bigg\{ \int_{\Gamma}w(t, y)dm^1(t)(y)-\int_{\Gamma}r(t, x)dm^2(t)(x):\nonumber\\
	   \text{$w, r$ continuous and $w(t, y)-r(t, x)\leq c(x, y)$ for every }x, y\in\Gamma\Bigg\}.
\end{align}
Hence, for any $t\in(0,T)$, the equilibrium condition in the labour market is equivalent to find a pair of continuous functions $(\th_1,\th_2)=( w(t, \cdot), -r(t, \cdot),)$ satisfying $\th_1( x)+\th_2( y)\leq c(x, y)$ and optimal for the dual problem \eqref{dualformulationmodel}.
\par
\bigskip
\noindent
\textbf{The  Mean-Field Game-Optimal Transport System.} The optimal control problems solved by the two populations are  coupled through the rent costs $R^1$ and $R^2$  and the     potentials $\th_1$ and $\th_2$ of the optimal transport problem. The   necessary conditions for equilibria  can be characterized by a Mean-Field Game system coupled with the optimality conditions for the transport system.\\
Associated with the Langrangian $L^\iota$,  $\iota=1,2$, of workers and firms, we introduce the Hamiltonians $H^\iota:(\cup_{\a\in\cA}\Gamma_{\a}\setminus\cV)\times\R\lra\R$ which are defined on each edge by   
\begin{equation}\label{hamiltonian}
	H_\a^\iota(x, p)=\sup_{u \in U^\iota_{\a}}\{-up-L_\a^\iota(x, u)\},\quad x\in\Gamma_{\a}\setminus\cV,\ p\in\R.
\end{equation}
The Mean-field Game-Optimal Transport {(MFGOT)} problem reads as
\begin{enumerate}
\item[$(i)$] {\it Forward-Backward MFG system:} for $(t, x)\in (0, T)\times(\Gamma_{\a}\setminus\cV),\ \a\in\cA$, $\iota=1,2$,
\begin{align*}
	&-\partial_t\phi_\iota-\mu^\iota_{\a}\partial^2\phi_\iota+H^\iota(x, \partial\phi_\iota)=R^\iota[m_1(t),m_2(t)]+\th_\iota,\qquad   \\
	&\partial_tm_\iota-\mu^\iota_{\a}\partial^2m_\iota-\partial(m_\iota\partial_pH^\iota(x, \partial\phi_\iota))=0.
\end{align*}
\item[$(ii)$] {\it Transition conditions: }for $(t, \nu_j)\in (0, T)\times\cV$, $\a,\b\in\cA$,  $\iota=1,2$,
\begin{align*}
	&\sum_{\a\in\cA_j}\gamma^\iota_{j\a}\mu^\iota_{\a}\partial_{\a}\phi_\iota(t, \nu_j)=0,\qquad &\\
	&\sum_{\a\in\cA_j}\mu^\iota_{\a}\partial_{\a}m_\iota(t, \nu_j)+n_{j\a}\partial_{p}H_\a^\iota\left(\nu_j, \partial\phi_\iota(t, \nu_j)\right)m_\iota|_{\Gamma_{\a}}(t, \nu_j)=0,\qquad&\\
	&\phi_\iota|_{\Gamma_{\a}}(t, \nu_j)=\phi_\iota|_{\Gamma_{\beta}}(t, \nu_j), \,\frac{m_\iota|_{\Gamma_{\a}}(t, \nu_j)}{\gamma^\iota_{j\a}}=\frac{m_\iota|_{\Gamma_\beta}(t, \nu_j)}{\gamma^\iota_{j\beta}}\, ,
\end{align*}
($n_{j\a}=1$ if  $\nu_j =\pi_{\alpha} (\ell_\alpha)$ 
and $n_{j\a}=-1$ if $\nu_j =    \pi_{\alpha} (0)$).
\item[$(iii)$] {\it Initial-terminal conditions:} for $x\in\Gamma$, $\iota=1,2$,
\begin{align*}
	\phi_\iota(T, x)=0, \ \ m_\iota(0, x)=m_0^\iota,\quad x\in\G.
\end{align*}
\item[$(iv)$] {\it Optimal Transport problem:} for $t\in (0, T)$,
\begin{align*}
	&\th_1(t, x)+\th_2(t, y)\leq c(x, y)\quad \text{for all $(x,y)\in\G$ and} \\
	& \cC(m_1(t), m_2(t))= \int_{\Gamma}\th_1(t, x)dm_1(t)(x)+\int_{\Gamma}\th_2(t, y)dm_2(t)(y).
\end{align*}
\end{enumerate}
A solution to {(MFGOT)} system is given by two triples $$(\phi_\iota(t, x), m_\iota(t, x), \th_\iota(t, x)), \quad \iota=1,2,$$ satisfying $(i)$-$(iv)$ in a suitable sense (see the next section), where $\phi_\iota(x,t)$ represents the value function for an agent of population $\iota$ at position $x$ and time $t$,  $m_\iota(x,t)$   the corresponding distribution  of agents at $(x,t)$ and $(\th_1,\th_2)=(-r(t, \cdot), w(t, \cdot))$ are the opposite of the revenue for workers and the wage for firms at time $t$.\\ 
The problem (MFGOT) is composed by a family of  Hamilton-Jacobi equations and a Fokker-Planck equations, defined on each edge $\G_\a$, see $(i)$. 
The equations defined  in the edges communicate through the transition condition at the vertices in $(ii)$, a Kirchhoff condition for the Hamilton-Jacobi equation and a conservation of the flux condition for the Fokker-Planck equation. Moreover the functions $\phi_\iota$, $\iota=1,2$, are assumed to be continuous at the vertices, while this not necessarily holds for $m_\iota$, whose condition expresses a partition law for the distribution of the density at the vertices,  see in $(iii)$. Finally, for any $t\in (0,T)$, in $(iv)$ an equilibrium condition given by an optimal transport problem is imposed. 
The Hamilton-Jacobi equation and the Fokker-Planck equation are coupled through the cost $R^\iota[m_1,m_2]$ in the former equation and the optimal control $\partial_pH^\iota(x, \partial\phi_\iota)$ in the latter one. Moreover the MFG system and Optimal Transport problem are coupled in the HJ equations through the Kantorovich potential $\th_\iota$, which in turn depends on $m_\iota$. 
Mean Field Games on networks have been considered in \cite{achdou,camillimarchi_sicon,camillimarchi}.

\begin{remark}[Modeling Remarks]
Several simplifying assumptions are introduced to maintain mathematical tractability while preserving the model’s economic relevance.
The hypothesis that workers and firms continuously adjust their locations should be interpreted as an \emph{idealized representation of long-run mobility} or \emph{gradual adaptation through search and matching processes}, rather than as literal instantaneous relocation. Accordingly, the time variable is meant to be understood at a coarse temporal scale (e.g., years rather than days), so that continuous spatial adjustment captures the aggregate outcome of many small and infrequent relocation or job-search decisions. This interpretation differs with the instantaneous time scale underlying the optimal transport problem, which models the short-run optimization of commuting patterns between residences and workplaces.

The choice to treat commuting costs as exogenous serves to isolate the role of spatial geometry and network connectivity in shaping equilibrium configurations. Nevertheless, the framework could readily accommodate endogenous congestion effects by allowing commuting costs to depend on local population densities, as is standard in urban transport models \cite{SmallVerhoef2007}.
Finally, the model is not intended to describe short-run relocation dynamics or discrete moving decisions at the individual level; rather, its scope is confined to long-run spatial equilibria and gradual adjustment paths.

The \emph{forward-looking behavior of workers}, modeled through stochastic control, represents an \emph{expected-utility optimization under uncertainty}, consistent with the notion of rational expectations. 
Alternative formulations based on \emph{bounded rationality} \cite{Simon1955,Manski1993,camillimarchi} could also be incorporated in future extensions, introducing more behavioral realism without altering the structure of the Mean-Field Game system. 

From an economic perspective, the \emph{rent terms} \( R^1[m_1,m_2] \) and \( R^2[m_1,m_2] \) formalize the \emph{competition for land and space} that lies at the core of classical urban economics models \cite{Alonso1964,Mills1967,Muth1969,GlaeserGyourko2005}. 
In our formulation, these functions act as the spatial analog of land or housing rents that adjust endogenously to local congestion and population density. 
Moreover, the coupling of the two populations through the \emph{Optimal Transport problem} introduces a \emph{matching mechanism} between workers and firms: each equilibrium transport plan \( \gamma \) identifies an optimal allocation of labor in space, analogous to the \emph{search-and-matching frameworks} of labor economics \cite{ShimerSmith2000}. 
This perspective highlights that, beyond its mathematical formulation, the model provides a bridge between \emph{classical urban land-use theory} and \emph{modern spatial matching models}.
\end{remark}

\section{Existence and Uniqueness}\label{sec:existence}
The aim of this section is to prove an existence and uniqueness result for  
solution to the {(MFGOT)} system. We start recalling some functional spaces defined on the network. The space
$C(\Gamma)$ is composed of the continuous functions on $\Gamma$ and,  
for $m\in\mathbb{N}$, we set
\[
C^{m}\left(\Gamma\right):=\left\{ v\in C\left(\Gamma\right):v_{\alpha}\in C^{m}\left(\left[0,\ell_{\alpha}\right]\right)\text{ for all }\alpha\in\cA\right\} ,
\]
  endowed with the norm
$$ 
\left\Vert v\right\Vert _{C^{m}\left(\Gamma\right)}:= {\sum}_{\alpha\in\cA}{\sum}_{k\le m}\left\Vert \partial^{k}v_{\alpha}\right\Vert _{L^{\infty}\left(0,\ell_{\alpha}\right)}.
$$
We define
\begin{align*}
	&H^m(\G):=\left\{ v:\G \lra\R:\, v\in C\left(\Gamma\right)\, \text{and}\,v_{\alpha}\in  H^m\left((0,\ell_{\alpha})\right)\text{ for all }\alpha\in\cA\right\},\\
	&H^{m}_{b}(\Gamma):=\left\{v:\Gamma\lra\R:\,v_{\alpha}\in H^m((0,\ell_{\alpha}))\text{ for all }\alpha\in\cA\right\}.
\end{align*}
The previous spaces are endowed with the standard norm
$$
\left( \sum_{\alpha\in\cA}\Vert   v_{\alpha} \Vert _{H^m(0,\ell_\a)}^{2}\right)^{1/2}.
$$
We set $V=H^1(\G)$, $V'=H^{-1}(\G)$ and $\langle \cdot,\cdot\rangle_{V', V}$  for the corresponding pairing. We also set
\begin{align*}
&W:=\left\{w\in H_{b}^{1}\left(\Gamma\right):\,\frac{w|_{\Gamma_{\alpha}}(\nu_{j})}{\gamma_{j\alpha}}=\frac{w|_{\Gamma_{\beta}}(\nu_{j})}{\gamma_{j\beta}}\text{ for all }j\in I,\ \alpha,\beta\in\cA_{j}\right\}, \\
&PC:=\left\{v:[0,T]\times\G\lra \R:\,v|_{[0,T]\times\G_{\alpha}}\in C([0,T]\times\G_{\a})\, \text{ for all $\alpha\in \cA$}\right\}.	
\end{align*}
For a  function  $v$ either in  $W$ or in $PC$, for which continuity at the vertices is not required, we still denote with $v_\alpha$ the extension by continuity of $v_\alpha$ on the whole interval $[0,\ell_\alpha]$. \\
The couple $(\G,d_\G)$, where $d_\G$ is the geodesic distance on the network, is a metric space. We denote with $\cM$ the space of Borel probability measures on $\G$ endowed with the topology of weak convergence. For $1 \le p < \infty$, the $L^p$-Wasserstein distance $\Wass_p$ between $m_1,m_2 \in\cM$ is defined by the Monge-Kantorovich optimal transport problem
\begin{align*}
	\Wass_p(m_1,m_2)=\min_{\gamma\in\Pi(m_1,m_2)}\left\{\int_{\G\times\G}(d_\G(x,y))^p d\gamma(x,y) \right\},
\end{align*}
where  $\Pi(m_1,m_2)$ denotes the set of transport plans, i.e. Borel probability measures on $\G\times\G$ with marginals $m_1$ and $m_2$. Since $\G$ is compact, the Wasserstein distance $\Wass_p$ metrises the topology of weak convergence of probability measures on $\G$. In particular, for $p=1$, we have the dual formula
\begin{multline*}
	{\Wass}_1(m_1,m_2)\\
=\sup\left\{\int_\G f(x)d(m_1-m_2):\, f:\G\lra\R,\,|f(x)-f(y)|\le d_\G(x,y)\right\}.
\end{multline*}
Given a continuous, non negative function $c:\Gamma\times\Gamma \lra\R$ and $m_1,m_2\in \cM$, consider the Optimal Transport problem
\begin{equation*}
	 \cC(m_1, m_2)=\inf_{\gamma\in\Pi(m_1, m_2)}\int_{\Gamma\times\Gamma}c(x, y)d\gamma(x, y).
\end{equation*}
The $c$-{\it transform} of a function $\phi:\Gamma\lra\R\cup\{+\infty\}$ is defined by $\phi^c(y):=\inf_{x\in\Gamma}\{c(x, y)-\phi(x)\}$ and a  function $\psi:\Gamma\lra\R\cup\{+\infty\}$ is called $c$-{\it concave} if there exists   $\phi$ such that $\psi=\phi^c$. 
Since $\G$ is compact, by Kantorovich duality theorem (see e.g. \cite{santambrogio,mugnolo}), we have the following identity  
\begin{multline*}
\cC(m_1, m_2)\\
=\sup_{\phi, \psi\in C(\Gamma)}\left\{\int_{\Gamma}\phi dm_1+\int_{\Gamma}\psi dm_2:\phi(x)+\psi(y)\leq c(x, y),\ \forall x,y\in\Gamma\right\}.
\end{multline*}
Moreover, the supremum is attained by a maximizing pair of the form $(\phi, \psi)=(\phi, \phi^c)$, where $\phi$ is a $c$-concave function. A  maximizer $\phi$, which in general is not unique, is called a {\it Kantorovich potential}.\\
Now we state the assumptions which we will assume in the rest of the paper. We assume that the  initial distributions of the agents satisfy
\begin{equation}\label{initial_cond}	
m_0^\iota\in L^2(\Gamma)\cap\cM, \ \ m_0^\iota\geq\d>0, \ \ \int_{\Gamma}m_0^\iota(x)dx=1, \ i=1,2,
\end{equation}
for some $\d>0$. Moreover, we  assume that  the Hamiltonian $H_\a^\iota(\cdot, p)$, $\iota=1,2$, defined in \eqref{hamiltonian}, satisfies
\begin{equation}\label{hyp:H}
	\begin{cases}
			  (i) &\text{$H_\a^\iota\in C^1(\Gamma_{\a}\times\R)$},\\
 			 (ii) &\text{$H_\a^\iota(x, \cdot)$ is convex in $p$ for any $x\in\Gamma_{\a}$},\\
 			 (iii)& \text{$H_\a^\iota(x, p)\leq C_0(|p|+1)$ for any $(x, p)\in\Gamma_{\a}\times\R$},\\ 
 			  (iv)&  \text{$|\partial_pH_\a^\iota(x, p)|\leq C_0$ for any $(x, p)\in\Gamma_{\a}\times\R$,}\\
 			 (v)  & \text{$|\partial_xH_\a^\iota(x, p)|\leq C_0$ for any $(x, p)\in\Gamma_{\a}\times\R$,}
	\end{cases}
\end{equation}
for a constant $C_0$ independent of $\a$. We also assume that the viscosity and the Kirchhoff coefficients satisfy 
\begin{equation*} 
	\mu^\iota_\a>0,  \,\g^\iota_{j,\a}>0, \, \sum_{\alpha\in\cA_{j}}\gamma^\iota_{j\alpha}\mu^\iota_{\alpha}=1, \quad  \a\in \cA,\,j\in J,\, \iota=1,2;
\end{equation*}
 the commuting cost satisfies
\begin{equation}\label{comm_cost}
c\in C^1(\G\times \G);
\end{equation}\\
the coupling  costs $R^\iota$, $\iota=1,2$, are continuous and uniformly  bounded in $\cM\times\cM\times\G$ and
\begin{equation}\label{ip_coupling} 
\max_{\a\in\cA}\max_{x\in\G_\a}\left|R^\iota_\a[m_1, m_2](x)-R^\iota_\a[\eta_1,\eta_2](x)\right|\le L\max_{\iota=1,2}\Wass_1(m_\iota,\eta_\iota),
\end{equation}  
with $m_\iota,\,\eta_\iota\in \cM,$ $\iota=1,2$.

\begin{remark}
Conditions $(ii)$-$(v)$ in \eqref{hyp:H}  are related to the boundedness and
Lipschitz continuity of the dynamic and the running cost and   are satisfied by a standard control-type Hamiltonian. Condition $(i)$ can be ensured by taking, for example, a Hamiltonian quadratic in \(p\) near the origin and smoothly linearly extended for large $|p|$, e.g., $H_\alpha^\iota(x,p) = C_0(|p|+1)$ for $|p|$ sufficiently large. This choice guarantees that all conditions in \eqref{hyp:H} are satisfied.
\end{remark}

\begin{definition}
	\label{solmfg}
	A solution to the (MFGOT) problem   is given by two triples $(\phi_\iota, m_\iota, \th_\iota)$, $\iota=1,2$, such that 
	\begin{enumerate}
		\item[(i)] 
	$\phi_\iota\in L^2((0, T); H^2(\Gamma))\cap C([0, T]; H^1(\G))$, $\partial_t\phi_\iota\in L^2((0, T)\times\Gamma)$, $\phi_\iota(T, x)=0$ and	 
	\begin{align*}
		\int_{\Gamma}[-\partial_t\phi_\iota w+\mu^\iota\partial \phi_\iota\partial w+H^\iota(x, \partial \phi_\iota)w]dx=\int_{\Gamma}[R^\iota[m_1(t),m_2(t)]+\th_\iota]wdx,
	\end{align*}
for all $w\in W$, a.e. in $t\in(0, T)$;	
	\item[(ii)] $m_\iota\in L^2((0, T); W)\cap C([0, T]; L^2(\Gamma)\cap\cM)$, $\partial_tm_\iota\in L^2((0, T); V')$, \\
	$m_\iota(0, x)=m_0^\iota$ and 
 	\begin{align*}
	\langle \partial_tm_\iota, v\rangle_{V', V}+\int_{\Gamma}\mu^\iota\partial m_\iota\partial vdx+\int_{\Gamma}\partial_pH^\iota(x, \partial\phi_\iota)m_\iota\partial vdx=0
	\end{align*}
for all $v\in V$, a.e. in $t\in(0, T)$;	
\item[(iii)] for any $t\in [0,T]$, $(\th_1(t, \cdot), \th_2(t, \cdot))=(\th(t, \cdot), \th^c(t, \cdot))$ where   $\th(t, \cdot)$ is  a $c$-concave Kantorovich potential, i.e. 	 
	\begin{equation}\label{kantorovich_pot}
		 \cC(m_1(t), m_2(t))=\int_{\Gamma}\th(t, x)dm_1(t)(x)+\int_{\Gamma}\th^c(t, y)dm_2(t)(y),
	\end{equation}
such that $\int_\G \th(t, x)dx=0$.
\end{enumerate}
\end{definition}
We state an existence and uniqueness result for the MFG system.
\begin{theo}
	 There exists a  solution  to the {(MFGOT)} system. Moreover, if 
	 \begin{equation}
       \label{monotonecoupling}
	 \sum_{\iota=1}^2	\int_{\Gamma}(R^\iota[m_1,m_2]-R^\iota[\bm_1,\bm_2])(m_\iota-\bm_\iota)dx\ge 0,
	 \end{equation}
	 for any $(m_1,m_2)$, $(\bm_1,\bm_2)\in \cM\times\cM$, with the equality implying
	 $R^\iota[m_1,m_2]=R^\iota[\bm_1,\bm_2]$ for $\iota=1,2$,
	 then the solution is unique.
\end{theo}
\begin{proof}
\ 
\par\smallskip\noindent
	{\it Existence.}	
	Set
	$$
	X=\left\{m\in C([0, T], \cM):\Wass_1(m(t), m(s))\leq {\bar C}|t-s|^{\frac{1}{2}},\ m(t)\geq\d_0>0\right\},
	$$
	where $\d_0$ is as in \eqref{initial_cond} and ${\bar C}$ will be fixed later.  Since $\G$ is bounded,  $X$ is a convex, compact subset of $C([0,T], \cM)$. Define a map $\cT:X^2\lra X^2$ in the following way:
	\begin{itemize}
\item[$(i)$] Given $(m_1, m_2)\in X^2$, let $\th_1, \th_2:[0, T]\times\Gamma\lra\R$ be such that	$(\th_1,\th_2)=(\th(t, \cdot), \th^c(t, \cdot))$ with the function $\th(t, \cdot)$ a $c$-concave Kantorovich potential for $ \cC(m_1(t), m_2(t))$ for $t\in [0,T]$, see \eqref{kantorovich_pot}, satisfying $\int_\G\th(t, x) dx=0$;\\  
\item[$(ii)$] For $(\th_1, \th_2)$ as before, let $\phi_\iota$ be the solutions of the HJ equations
	\begin{equation}
		\label{hj}
		\left\{
		\begin{array}{ll}
			-\partial_t\phi_\iota-\mu^\iota_{\a}\partial^2\phi_\iota+H_\a^\iota(x, \partial\phi_\iota)=&R^\iota[m_1,m_2]+\th_\iota,\\
			&(t, x)\in(0, T)\times(\Gamma_{\a}\setminus\cV),\ \a\in\cA,\\[6pt]
			\sum_{\a\in\cA_j}\gamma^\iota_{j\a}\mu^\iota_{\a}\partial_{\a}\phi_\iota(t, \nu_j)=0, & (t, \nu_j)\in(0,T)\in\cV,\\[6pt]
			\phi_\iota|_{\Gamma_{\a}}(t, \nu_j)=\phi_\iota|_{\Gamma_{\beta}}(t, \nu_j),&  \a,\b\in\cA_j,\ (t, \nu_j)\in(0,T)\times\cV,\\[6pt]
			\phi_\iota(T, x)=0,&  x\in\Gamma;
		\end{array}
		\right.
	\end{equation}
\item[$(iii)$] Set $(\eta_1, \eta_2)=\cT(m_1, m_2)$, 	where $\eta_\iota$, $\iota=1,2$, is the solution of
	\begin{equation}
		\label{fp}
		{\small
	\left\{
			\begin{array}{ll}
			\partial_t\eta_\iota-\mu^\iota_{\a}\partial^2\eta_\iota-\partial(\eta_\iota\partial_pH_\a^\iota(x, \partial\phi_\iota))=0,&\\
			&(t, x)\in(0, T)\times(\Gamma_{\a}\setminus\cV),\ \a\in\cA,\\[6pt]
			\sum_{\a\in\cA_j}\mu^\iota_{\a}\partial_{\a}\eta_\iota(t, \nu_j)+n_{j\a}\partial_{p}H_\a^\iota \left(\nu_j, \partial\phi_\iota(t, \nu_j)\right)&\eta_\iota|_{\Gamma_{\a}}(t, \nu_j)=0,\\
			&(t, \nu_j)\in(0,T)\times\cV,\\[6pt]
			\dfrac{\eta_\iota|_{\Gamma_{\a}}(t, \nu_j)}{\gamma^\iota_{j\a}}=\dfrac{\eta_\iota|_{\Gamma_\beta}(t, \nu_j)}{\gamma^\iota_{j\beta}},&\a, \b\in\cA,\ (t, \nu_j)\in(0,T)\times\cV,\\[6pt]
			\eta_\iota(0, x)=m_0^\iota,&x\in\Gamma.
\end{array}
		\right.}
	\end{equation}
	\end{itemize}
We prove, by Schauder Fixed-Point Theorem (see e.g. \cite{gilbarg}, Corollary 11.2), that the  map $\cT$ has a fixed point and this gives a solution of the {(MFGOT)} system.\\
First observe that the function $\cT$ is well-defined. Indeed, given $(m_1, m_2)\in X^2$, by Proposition \ref{optimaltransportnetwork2}, \cite[Theorem 3.1]{achdou} and the normalization condition $\int_\G\th(t, x) dx=0$, the couple $(\th_1,\th_2)=(\th(t, \cdot), \th^c(t, \cdot))$ is uniquely defined for any $t\in (0,T)$. Moreover, by Proposition \ref{normkantorovich} and \ref{prop:stability_pot}, the functions $\th_\iota$ are continuous and uniformly bounded in $(0,T)\times\G$ and therefore $\th_\iota\in L^2((0,T)\times\G)$. Furthermore,  by \eqref{ip_coupling}, also $R^\iota\in L^2((0,T)\times\G)$ and therefore by \cite[Theorem 4.1]{achdou} it follows that there exists a  unique solution $\phi_\iota$, $\iota=1,2$ to \eqref{hj} in the sense of Definition \ref{solmfg}.$(i)$. Moreover, by \cite[Theorem 3.1]{achdou}, there exists a unique solution $\eta_\iota$, $\iota=1,2$, which solves \eqref{fp}    in the sense of Definition \eqref{solmfg}.$(ii)$. Since $\eta_\iota$ can be identified with the corresponding Borel measure with density $\eta_\iota(t)$ on $\Gamma$ at time $t$, by \cite[Theorem 3.1]{achdou} and \cite[Proposition 4.3]{camillimarchi}, choosing ${\bar C}$ equal to $C_{W}=C_{W}(\|\partial_pH\|_{L^{\infty}})$ in \cite[Proposition 4.3]{camillimarchi}, we have that $\cT$ maps $X$ into itself.\par
We prove that $\cT$ is continuous. Given $m_n=(m_{n,1},m_{n,2})\in X^2$  and $m=(m_1,m_2)\in X^2$ such that $\Wass_1(m_{n,\iota},m_\iota)\lra 0$ for $n\to \infty$, $\iota=1,2$,
let $(\th_{n,1},\th_{n,2})=(\th_n, \th_n^c)$ and $(\th_1,\th_2)=(\th,\th^c)$, where, for any $t\in (0,T)$, $\th_n$  and $\th$ are the  Kantorovich potentials corresponding to $ \cC(m_{n,1}(t), m_{n,2}(t))$ and $ \cC(m_1(t), m_2(t))$, renormalized in such a way that $\int_\G\th_n dx=\int_\G\th dx=0$ (recall that $\th_n$ and $\th$ are uniquely defined). Then let $\phi_{n,\iota}$ and $\phi_\iota$, $\iota=1,2$, be the solutions of the HJ equations \eqref{hj} with right-hand side $R^\iota[m_{n,1},m_{n,2}]+\th_{n,\iota}$ and $R^\iota[m_1,m_2]+\th_\iota$, respectively. Finally, set $(\eta_{n,1}, \eta_{n,2})=\cT(m_{n,1}, m_{n,2})$, $(\eta_1, \eta_2)=\cT(m_1, m_2)$. 
Since $m_{n,\iota}\lra m_\iota$ in $X$, then $\th_{n,\iota}(t, x)$ converges to $\th_\iota(t, x)$, uniformly in $x\in\G$ for any $t\in (0,T)$, see Proposition \ref{prop:stability_pot}. Moreover 
$\th_{n,\iota}$, $\th_\iota$ are uniformly bounded in $[0,T]\times\G$ and therefore $\th_{n,\iota}\lra \th_\iota$ in $L^2((0,T)\times\G)$. By \eqref{ip_coupling}, we also have  $R^\iota[m_{n,1}, m_{n,2}] \lra R^\iota[m_1, m_2] $  in $L^2((0, T)\times\Gamma)$. \\
By \cite[Lemma 4.1]{achdou}, we have that the sequence $\phi_{n,\iota}$ converges to $\phi_\iota$ in the space $L^2((0, T); H^2(\Gamma))\cap C([0, T]; V)\cap W^{1, 2}((0, T); L^2(\Gamma))$ . Therefore, \cite[Theorem 3.1]{achdou} with $b_n=\partial_pH^\iota(x, \partial\phi_{n,\iota})$ and  $b=\partial_pH^\iota(x, \partial\phi_\iota)$ implies that $(\eta_\iota)_n$ converges to $\eta_\iota$ in $L^\infty\left((0,T);L^2(\Gamma)\right)$, and then 
$$\Wass_1((\eta_\iota)_n, \eta_\iota)\lra 0\quad \text{for $n\to \infty$, $\iota=1,2$}.$$ Hence $\cT$ is continuous. Since $X^2$ is a convex and compact set, by Schauder Theorem the map $\cT$ has a fixed point, and hence the {(MFGOT)} system admits a solution.

\vskip 12pt
	{\it Uniqueness:}  We assume that there exist two solutions $(\phi_\iota, m_\iota, \theta_\iota)$ and  $(\psi_\iota, \eta_\iota, \rho_\iota)$, $\iota=1, 2$, of {(MFGOT)}. We set $\bar\phi_\iota=\phi_\iota-\psi_\iota$, $\bar m_\iota=m_\iota-\eta_\iota$ and $\bar \theta_\iota=\theta_\iota-\rho_\iota$ and we write the conditions $(i)$, $(ii)$ and $(iii)$ of {(MFGOT)} for $\bar\phi_\iota$, $\bar m_\iota$, $\bar \theta_\iota$:\\
$\bullet$ for $(t, x)\in (0, T)\times(\Gamma_{\a}\setminus\cV),\ \a\in\cA$, $\iota=1,2$,
$$-\partial_t\bar\phi_\iota-\mu^\iota_{\a}\partial^2\bar\phi_\iota+H^\iota(x, \partial\phi_\iota)-H^\iota(x, \partial\psi_\iota)-(R^\iota[m_1, m_2]-R^\iota[\eta_1, \eta_2])- \bar\theta_\iota=0, $$
$$\partial_t\bar m_\iota-\mu^\iota_{\a}\partial^2\bar m_\iota-\partial(m_\iota\partial_pH^\iota(x, \partial\phi_\iota)-\eta_\iota\partial_pH^\iota(x, \partial\psi_\iota))=0;$$
$\bullet$ for $(t, \nu_j)\in (0, T)\times\cV$, $\a,\b\in\cA$, $\iota=1,2$,
$$ \sum_{\a\in\cA_j}\gamma^\iota_{j\a}\mu^\iota_{\a}\partial_{\a}\bar\phi_\iota(t, \nu_j)=0,$$
\begin{multline*}
\sum_{\a\in\cA_j}n_{j\a}\left[m_\iota|_{\Gamma_{\a}}(t, \nu_j)\partial_{p}H_\a^\iota\left(\nu_j, \partial\phi_\iota|_{\Gamma_{\a}}(t, \nu_j)\right)\right.\\
\left.-\eta_\iota|_{\Gamma_{\a}}(t, \nu_j)\partial_pH_\a^\iota\left(\nu_j, \partial\psi_\iota|_{\Gamma_{\a}}(t, \nu_j)\right)\right]+\sum_{\a\in\cA_j}\mu^\iota_{\a}\partial_{\a}\bar m_\iota(t, \nu_j)=0,
  \end{multline*}
$$\bar\phi_\iota|_{\Gamma_{\a}}(t, \nu_j)=\bar\phi_\iota|_{\Gamma_{\beta}}(t, \nu_j), \,\frac{\bar m_\iota|_{\Gamma_{\a}}(t, \nu_j)}{\gamma^\iota_{j\a}}=\frac{\bar m_\iota|_{\Gamma_\beta}(t, \nu_j)}{\gamma^\iota_{j\beta}}\, ,$$
 $\bullet$ for $x\in\Gamma$, $\iota=1,2$,
\begin{align*}
	\bar\phi_\iota(T, x)=0, \ \ \bar m_\iota(0, x)=0,\quad x\in\G;
\end{align*}
$\bullet$ for any $t\in [0,T]$,
\begin{align*} 
	 \cC(m_1(t), m_2(t))=\int_{\Gamma}\th_1(t, x)dm_1(t)(x)+\int_{\Gamma}\th_2(t, y)dm_2(t)(y) ,\\
	 \cC(\eta_1(t), \eta_2(t))=\int_{\Gamma}\rho_1(t, x)dm_1(t)(x)+\int_{\Gamma}\rho_2(t, y)dm_2(t)(y).
\end{align*}
\par
{ Multiply by $\bar m_\iota\in W$ the PDE satisfied by $\bar\phi_\iota$, i.e.,
$$
-\partial_t\bar\phi_\iota-\mu^\iota_{\a}\partial^2\bar\phi_\iota+H^\iota(x, \partial\phi_\iota)-H^\iota(x, \partial\psi_\iota)-(R^\iota[m_1, m_2]-R^\iota[\eta_1, \eta_2])- \bar\theta_\iota=0,
$$
by $\bar\phi_\iota\in V$ the one satisfied by $\bar m_\iota$, i.e.,
$$
\partial_t\bar m_\iota-\mu^\iota_{\a}\partial^2\bar m_\iota-\partial(m_\iota\partial_pH^\iota(x, \partial\phi_\iota)-\eta_\iota\partial_pH^\iota(x, \partial\psi_\iota))=0,
$$
subtract the resulting equations and sum for $\iota=1,2$. Integrating the resulting identity in $\Gamma\times[0, T]$ and exploiting the conditions at the vertices, we obtain}
\begin{multline}
\label{uniqueness_sum}
\sum_{\iota=1}^2\left[\int_0^T\int_{\Gamma}(R^\iota[m_1, m_2]-R^\iota[\eta_1, \eta_2])(m_\iota-\eta_\iota)dxdt\right.\\
\left.+\int_0^T\int_{\Gamma}(\theta_\iota-\rho_\iota)(m_\iota-\eta_\iota)dxdt+\int_0^T\int_{\Gamma}\partial_t(\bar m_\iota\bar\phi_\iota)dxdt\right. \\
\left.+\sum_{\a\in\cA_j}\int_0^T\int_{\Gamma_{\a}}m_\iota[H^\iota(x, \partial\psi_\iota)-H^\iota(x, \partial\phi_\iota)-\partial_pH^\iota(x, \partial\psi_\iota)\partial\bar\phi_\iota]dxdt\right.\\
\left.+\sum_{\a\in\cA_j}\int_0^T\int_{\Gamma_{\a}}\eta_\iota[H^\iota(x, \partial\phi_\iota)-H^\iota(x, \partial\psi_\iota)+\partial_pH^\iota(x, \partial\phi_\iota)\partial\bar\phi_\iota]dxdt\right]=0.
\end{multline}
We claim that all the terms in left hand side of the previous identity are non negative. By \eqref{monotonecoupling}, the first term is non negative. We show that second term is non negative, i.e.
\begin{equation}
\label{monotonecoupling2}
\sum_{\iota=1}^2\int_0^T\int_{\Gamma}(\theta_\iota-\rho_\iota)(dm_\iota(t)-d\eta_\iota(t))dt\geq0.
\end{equation}
Observe that $(\theta_1(t, \cdot), \theta_2(t, \cdot))=(\th(t, \cdot), \th^c(t, \cdot))$ and also 
$(\rho_1(t, \cdot), \rho_2(t, \cdot))=(\rho(t, \cdot), \rho^c(t, \cdot))$, 
 beside  being   Kantorovich potential for $ \cC(m_1(t), m_2(t))$ and respectively $ \cC(\eta_1(t), \eta_2(t))$, satisfy
\begin{align*}
	 \cC(\eta_1(t), \eta_2(t))\ge \int_{\Gamma}\th(t, x)dm_1(t)(x)+\int_{\Gamma}\th^c(t, y)dm_2(t)(y),
 \\
	 \cC(m_1(t), m_2(t))\ge \int_{\Gamma}\rho(t, x)dm_1(t)(x)+\int_{\Gamma}\rho^c(t, y)dm_2(t)(y).
\end{align*}
Hence, for any $t\in (0,T)$,
\begin{align*}
&\int_{\Gamma}(\theta_1-\rho_1)(dm_1(t)-d\eta_1(t))(x)+\int_{\Gamma}(\theta_2-\rho_2)(dm_2(t)-d\eta_2(t))(y)\\
&= \cC(m_1(t), m_2(t))+ \cC(\eta_1(t), \eta_2(t))\\
&-\left(\int_{\Gamma}\theta_1d\eta_1(t)(x)+\int_{\Gamma}\theta_2d\eta_2(t)(y)+\int_{\Gamma}\rho_1dm_1(t)(x)+\int_{\Gamma}\rho_2dm_2(t)(y) \right)\\
&\geq  \cC(m_1(t), m_2(t))+ \cC(\eta_1(t), \eta_2(t))- \cC(\eta_1(t), \eta_2(t))- \cC(m_1(t), m_2(t))=0.
\end{align*}
Moreover
$$
\int_0^T\int_{\Gamma}\partial_t(\bar m_\iota\bar\phi_\iota)dxdt=\int_{\Gamma}[\bar m_\iota(T, x)\bar\phi_\iota(T, x)-\bar m_\iota(0, x)\bar\phi_\iota(0, x)]dx=0,
$$
since $\bar\phi_\iota(T, x)=0$ and $\bar m_\iota(0, x)=0$. Finally, since $H^\iota$ is convex and $m_\iota$, $\eta_\iota$ are non negative, the last two terms in \eqref{uniqueness_sum} are non negative. Therefore, the claim is proved and all the terms in the identity must vanish.\\
By \eqref{monotonecoupling2}, we obtain $R^\iota[m_1, m_2]=R^\iota[\eta_1, \eta_2]$ and therefore, by \eqref{monotonecoupling}, $(m_1, m_2)=(\eta_1, \eta_2)$ and also  $\theta_\iota=\rho_\iota$, $\iota=1,2$. Finally, by \cite[Theorem 4.1]{achdou}, we get $\phi_\iota=\psi_\iota$, $\iota=1, 2$.
\end{proof}
\begin{remark}
The assumption $m_0^\iota\ge\d>0$ in \eqref{initial_cond} is necessary to have that the support of $m_\iota(t)$, $t\in (0,T)$, coincides with the network $\G$ in order to guarantee the uniqueness, up to renormalization, of the Kantorovich potential corresponding to $ \cC(m_1(t),m_2(t))$. Indeed, by Proposition \ref{optimaltransportnetwork2}, this assumption can be relaxed assuming only that  the initial distribution of  one of the two populations, for example the  workers, is supported in the whole $\G$.	
\end{remark}

\section{Approximation and Numerical Experiments}\label{sec:discretization}

\subsection[Num. approx. of H-J and F-P eq. on a network via semi-Lagrangian methods]{Numerical approximation of Hamilton-Jacobi and Fokker-Planck e\-qua\-tions on a network via semi-Lagrangian methods}
In this section we describe the numerical schemes   for approximating the Ha\-mil\-ton-Jacobi and Fokker-Planck equations on a network. The schemes have been introduced, in the Euclidean case, in various works as \cite{CS17,carlini2016DGA}. The adaptation to networks, following the same principles used in \cite{carlini2020,camilliFesta}, has various non-trivial aspects. {Among the most interesting aspects of the proposed schemes is the extension of a semi-Lagrangian scheme on networks to the diffusive (i.e., second-order) case. We emphasize that this case is not covered by the aforementioned literature, and although it would be very interesting to investigate its numerical properties in terms of consistency and convergence in a possible future work, it has proven in practice to be an effective tool for approximating the solution, with the additional advantage of not requiring a restrictive condition of the temporal/spatial discretization step.}

Fixed $\Delta t,\,\Delta x>0$, we  consider a discretization $\Gamma_{\Delta x}$ of the network $\Gamma$ defined in the following way: given an edge $\Gamma_\alpha$   connecting vertices $\nu_i$, $\nu_j$, we consider the points
$$
	\pi_\alpha(y_k)=[y_k\nu_j+(\ell_\alpha-y_k)\nu_i]\ell_\alpha^{-1},\quad\text{for } y_k\in\{0=y_1,...,y_{N_\alpha}=\ell_\alpha\},
$$
where $y_{k+1}-y_k\leq\Delta x$.  We also define the unit vector $e_\alpha=(\nu_j-\nu_i)/\ell_\alpha$ and
$p^\iota_{j\alpha}$ as in   \eqref{coeff2}. We denote with $N=\sum_{\alpha\in \cA}N_\alpha$ the total number of discretization points in the network. 

\paragraph{Approximation of the Trajectories.}
First, we describe the approximation of the dynamics of the optimal control problems. A numerical version of the stochastic differential equations giving the dynamics of the agents is  obtained using the Euler-Maruyama scheme, taking care   that the discrete trajectories can exit the edge where they start  through a vertex and enter one of the adjacent edges. For the population $\iota\in\{1,2\}$, assuming   to know the optimal feedback map $u^\iota_\alpha:\Gamma\rightarrow [\underline {u^\iota_\alpha}, \overline {u^\iota_\alpha}]$ and that $X^\iota_s$ belongs to the edge $\Gamma_\alpha$ connecting the vertices $\nu_i$ and $\nu_j$, we set  
$$Y^\iota_{s+1}=X^\iota_s+\Delta t \, u^\iota_\alpha(X^\iota_s)e_\alpha,$$
and we consider the following cases
\begin{enumerate}
	\item
	If $Y^\iota_{s+1}\in\Gamma_\alpha$ and 
	\begin{itemize}
	\item[i)]   $\min\{|Y^\iota_{s+1}-\nu_i|,|Y^\iota_{s+1}-\nu_j|\} > \sqrt{2\Delta t \mu^\iota_\alpha}$, we define 
	$$ X^\iota_{s+1}=Y^\iota_{s+1}+(-1)^r \sqrt{2\Delta t \mu^\iota_\alpha}\,e_\alpha,$$
	with $r$ a random variable uniformly distributed in $\{0,1\}$.
	\item[ii)] If $|Y^\iota_{s+1}-\nu_j|< \sqrt{2\Delta t \mu^\iota_\alpha}$ (similarly if $|Y^\iota_{s+1}-\nu_i|< \sqrt{2\Delta t \mu^\iota_\alpha}$ ), we set $t^*=\inf\{t\in (0,\Delta t):\,Y^\iota_{s+1}+\sqrt{2t \mu^\iota_\alpha}e_\alpha n_{j\alpha}=\nu_j \}$, where	 $n_{j\alpha}=1$ if  $\nu_j =\pi_{\alpha} (\ell_\alpha)$ 
	and $n_{j\alpha}=-1$ if $\nu_j =    \pi_{\alpha} (0)$, and define
	$$
	 X^\iota_{s+1}=
	 \left\{ 
	 \begin{array}{ll}
	 	Y^\iota_{s+1}-\sqrt{2\Delta t \mu^\iota_\gamma}e_\alpha n_{j\alpha}& \hbox{with probability (w.p.) $1/2$},\\[4pt]
	 	\nu_j-\sqrt{2(\Delta t -t^*) \mu^\iota_\beta}\,e_\beta n_{j\beta}\quad &\hbox{w.p. $p^\iota_{j\beta}/2$, $\beta\in \cA_j$}.
	 \end{array}
	 \right.
	 	$$
	\end{itemize}
	\item
	If $Y^\iota_{s+1}\not\in\Gamma_\alpha$, set  $t^*=\inf\{t\in (0,\Delta t):\,X^\iota_s+t u^\iota_\alpha(X^\iota_s)e_\alpha\in\cV\}$ with $X^\iota_s+t^* u^\iota_\alpha(X^\iota_s)e_\alpha=\nu_j$ (similarly for $\nu_i$). We  assume that $X^\iota_{s+1}\in \Gamma_\beta$, $\beta \in \cA_j$, with probability $p^\iota_{j\beta}$ and, given $X^\iota_{s+1}\in \Gamma_\beta$,  we distinguish the following case
	\begin{itemize}
		\item[i)]  If $|(\Delta t-t^*) \, u_\beta(\nu_j)| > \sqrt{2(\Delta t-t^*) \mu^\iota_\beta}$, then we set
		$$ X^\iota_{s+1}= \nu_j+ (\Delta t-t^*)\, u_\beta(\nu_j)e_\beta+(-1)^r \sqrt{2(\Delta t-t^*) \mu^\iota_\beta}e_\beta, $$
		with $r$ a random variable uniformly distributed $\{0,1\}$.
		\item[ii)] If $|(\Delta t-t^*) \, u_\beta(\nu_j)| < \sqrt{2(\Delta t-t^*) \mu^\iota_\beta}$, we set 
		$$
		X^\iota_{s+1}=
		\left\{ 
		\begin{array}{ll}
			\nu_j+ (\Delta t-t^*)\, u_\beta(\nu_j)e_\beta-\sqrt{2(\Delta t-t^*) \mu^\iota_\beta}e_\beta n_{j\beta}& \hbox{w.p. $1/2$},\\[4pt]
			\nu_j-\sqrt{2(\Delta t -t^*) \mu^\iota_{\bar\beta}}\,e_{\bar\beta} n_{j{\bar\beta}} &\hbox{w.p. $p^\iota_{j{\bar\beta}}/2$,}\\
			& \hbox{ ${\bar\beta}\in \cA_j$}.
		\end{array}
		\right.
		$$
	\end{itemize}
\end{enumerate}

We write the description above in a compact form as 
$$ X^\iota_{s+1}=\Psi(X^\iota_s,u,\Delta t), $$
with the sense that the mean position of an agent is
$$ \mathbb E (X^\iota_{s+1})=\sum_{c=1}^{N^\iota_c} p^\iota_c \Psi_c(X^\iota_s,u,\Delta t), $$
with $p^\iota_c$ the probability of being in the sub-case $\Psi_c(X^\iota_s,u,\Delta t)$, and $N^\iota_c$ being the total number of possible positions assumed by the agent at time $s+1$. 

\paragraph{A Semi-Lagrangian Method for the Hamilton-Jacobi and the Fokker-Planck Equation.}
In the semi-Lagrangian method for a Hamilton-Jacobi equation we use a direct computation of all the possible trajectories and we optimize on the control the local expected value of the problem. It means that, calling $\Phi_\iota:\Gamma_{\Delta x}\times \{t_0,...,t_T\}$ the discrete approximation of $\phi_\iota$ such that $\Phi_\iota^{j,n}$ is relative to the population $\iota$ at the point $y=j\Delta x$ at time $t_n=n\,\Delta t$, the numerical scheme can be described as 
\begin{equation}\label{hj:scheme}
\Phi^{j,n+1}_\iota=\sum_{c=1}^{N_c} p_c \left(\max_{u^\iota} \mathbb I[\Phi_\iota^{n}](\Psi_c(j\Delta x,u,\Delta t))-L^\iota(j\Delta x,u)-\theta_\iota^{n,j}\right).
\end{equation}
Here, $\mathbb I$ is an opportune interpolation operator which allow us to evaluate the function $\Phi^{n}_\iota$ also in points not following exactly in the discrete network $\Gamma_{\Delta x}$, $L^\iota$ is the running cost of the population $\iota$ while $\theta_\iota$ is the Kantorovich potential of the optimal transportation problem whose calculation we will address below. 

We build the relative scheme for the Fokker-Planck equation taking the adjoint of the previous equation,  as described for the Euclidean case in \cite{CS17}. In this scheme, we will use the optimal feedback map $u^\iota_*$ found in \eqref{hj:scheme}. Calling $M:\Gamma_{\Delta x}\times \{t_0,...,t_T\}\rightarrow [0,+\infty)$ the discrete approximation of $m^\iota$ we have that 
\begin{equation*}
M_\iota^{j,n+1}=\sum_{c=1}^{N_c} p_c \sum_{k}\beta_{k}(\Psi_c(j\Delta x,u,\Delta t)) M_\iota^{k,n}.
\end{equation*}
Here, the $\beta_{k}$ are the basis of the polynomial approximation used in the interpolation operator $\mathbb I[\Phi](x)$ introduced above. A theoretical study of the introduced approximation scheme will not be carried out in this work, but postponed to a future work.\\

\subsection{Approximation of an Optimal Transportation Problem on a Network via Entropic Regularization}
We consider the discrete optimal transportation problem on a network, aiming to find the optimal way to move the  configuration of workers \(M^1:\Gamma_{\Delta x} \rightarrow [0,+\infty)\) to the one of firms \(M^2:\Gamma_{\Delta x} \rightarrow [0,+\infty)\), see \eqref{OT_primal}. We seek the optimal transfer plan (a matrix \(\pi\) such that \(\sum_{j}\pi_{i,j}=M^1_i\) and \(\sum_{i}\pi_{i,j}=M^2_j\)) that minimizes a given transfer cost \(c_{i,j}:\Gamma_{\Delta x} \times \Gamma_{\Delta x} \rightarrow [0,+\infty)\). Formally, the problem is defined as:

\begin{equation}\label{disc:opttransp}
\left\{
\begin{array}{l}
\min_{\pi \in \mathbb{R}^{N \times N}} \sum_{i,j} c_{i,j} \pi_{i,j} ,\\
\pi_{i,j} \geq 0, \quad \forall i,j \in \{1,\ldots,N\}, \\
\sum_{j} \pi_{i,j} = M^1_i, \quad \forall i \in \{1,\ldots,N\}, \\
\sum_{i} \pi_{i,j} = M^2_j, \quad \forall j \in \{1,\ldots,N\}.
\end{array}
\right.
\end{equation}

It is well-known that an optimal transport problem is related to its dual representation in the form of a Monge-Kantorovich problem, which also holds in the discrete setting \cite{solomon}. Therefore, solving problem \eqref{disc:opttransp} is equivalent to finding two discrete potentials \(\theta^1_i, \theta^2_i:\Gamma_{\Delta x} \rightarrow (-\infty,+\infty)\) such that

\[
\max_{\theta^1, \theta^2} \sum_{i=1}^N \left( \theta^1_i M^1_i + \theta^2_i M^2_i \right) \quad \text{subject to} \quad \theta^1_i + \theta^2_j \leq c_{i,j}.
\]
Recall that the Kantorovich   potential  need to be computed at  any time $t\in [0,T]$, see \eqref{kantorovich_pot}. Due to the large number of optimal transport problems that must be solved, classical numerical techniques like the fluid mechanics approach \cite{BB00} are not the most appropriate for our case. Instead, we accelerate the process using entropic regularization, a technique introduced by the machine learning community in \cite{Cuturi}. In this approach, we solve an entropic regularized version of the problem:
$$
\left\{
\begin{array}{l}
\min_{\pi \in \mathbb{R}^{N \times N}} \sum_{i,j} \left( \pi_{i,j} c_{i,j} + \sigma \pi_{i,j} \log \pi_{i,j} \right), \\
\sum_{j} \pi_{i,j} = M^1_i, \quad \forall i \in \{1,\ldots,N\}, \\
\sum_{i} \pi_{i,j} = M^2_j, \quad \forall j \in \{1,\ldots,N\},
\end{array}
\right.
$$
for a fixed parameter \(\sigma > 0\). Note that the constraint  \(\pi_{i,j} \geq 0\)  is implicitly satisfied because \(\log \pi_{i,j}\) in the objective function prevents negative \(\pi\) values. As \(\sigma \rightarrow 0^+\), we recover the original problem \eqref{disc:opttransp}. The Lagrangian of the optimization problem becomes
\begin{multline*}
\mathcal{L}(\pi, \theta^1, \theta^2) = \sum_{i,j} \pi_{i,j} c_{i,j}\\
 + \sigma \sum_{i,j} \pi_{i,j} \log \pi_{i,j} + \sum_{i} \theta^1_i \left( M^1_i - \sum_j \pi_{i,j} \right) + \sum_{j} \theta^2_j \left( M^2_j - \sum_i \pi_{i,j} \right) \\
= \langle \pi, C \rangle + \sigma \langle \pi, \log \pi \rangle + (\theta^1)^T (M^1 - \pi \mathbbm{1}) + (\theta^2)^T (M^2 - \pi^T \mathbbm{1}),
\end{multline*}
where \(C\) is the matrix with elements \(c_{i,j}\), \(\langle \cdot, \cdot \rangle\) denotes the element-wise inner product of matrices, \(\log\) is taken element-wise, and \(\mathbbm{1}\) is the vector of all ones. By computing the optimality conditions, we obtain:
\[
0 = \nabla_{\pi} \mathcal{L} = C + \sigma \mathbbm{1} \mathbbm{1}^T + \sigma \log \pi - \theta^1 \mathbbm{1}^T - \mathbbm{1} (\theta^2)^T.
\]
Solving for \(\pi\):

\[
\log \pi = \frac{\theta^1 \mathbbm{1}^T}{\sigma} + \frac{\mathbbm{1} (\theta^2)^T}{\sigma} - \frac{C}{\sigma} + \log K_\sigma,
\]
where the matrix  \(K_\sigma = e^{-C/\sigma}\). Letting \(u = e^{\theta^1/\sigma}\) and \(v = e^{\theta^2/\sigma}\), we find:

\[
\pi = \text{diag}(u) K_\sigma \text{diag}(v),
\]
where \(\text{diag}(p)\) denotes the diagonal matrix with \(p\) on its diagonal. The existence of \(u\) and \(v\) follows from Sinkhorn's Theorem \cite{Sinkhorn}, which guarantees that for any matrix \(A\) with positive entries, there exist diagonal matrices \(D_1\) and \(D_2\) such that \(A = D_1 \pi D_2\), with \(\pi\) being a doubly stochastic matrix. This theorem suggests an iterative algorithm to find \(u\) and \(v\) as the fixed points of:

$$
\left\{
\begin{array}{l}
u^{k+1} = \frac{M^1}{K_\sigma v^k}, \\
v^{k+1} = \frac{M^2}{K_\sigma^T u^{k+1}},
\end{array}
\right.
$$
where the divisions are element-wise. This algorithm alternately updates \(u\) and \(v\), converging asymptotically to the optimal \(\pi\) at an efficient rate, regardless of the initial guess. Once convergence to a pair \((\bar{u}, \bar{v})\) within a certain tolerance is reached, the potentials \(\theta^1\) and \(\theta^2\) can be obtained as:
$$
\theta^1 = \sigma \log \bar{u}, \quad \theta^2 = \sigma \log \bar{v}.
$$

\subsection{Numerical Experiments}
We develop a collection of tests by considering specific coupling costs inspired from \cite{abc}. The model is relatively standard regarding the choice of Hamiltonians, but it includes a structure in the coupling costs $R^\iota$ that is appropriate for simulating the processes of aggregation and segregation between the two populations. More specifically:

\par\medskip\noindent
\textbf{Lagrangian:} We consider a Lagrangian given by
$$
L^\iota_\alpha(x, p) = \frac{\rho^\iota_\alpha}{2}|p|^2 + \Lambda^\iota_\alpha(x), \qquad \iota=1,2, \alpha \in \cA,
$$
hence the Hamiltonian is 
$$
H^\iota_\alpha(x, p) = \frac{1}{2\rho^\iota_\alpha}|p|^2 + \Lambda^\iota_\alpha(x).
$$
Here, $\rho^\iota$ represents the mobility cost, which are low for workers and high for firms, thus $\rho^1 < \rho^2$. The function $\Lambda^\iota$ represents the housing cost of a given area, which may depend on the speed of connections for both populations and the presence of parks, shopping centers, or other infrastructure for workers. Note that, even if our approach does not allow to consider a quadratic Hamiltonian $H^\iota$ since  it does not satisfy the  assumptions   \eqref{hyp:H}, for simplicity we consider it in the numerical experiments.

\par\medskip\noindent
\textbf{Coupling Cost:} The coupling cost is of the form
$$
R^\iota[m_1, m_2] = S^\iota(m_1, m_2) + O^\iota[m_1, m_2], \qquad \iota=1,2.
$$
\begin{itemize}
    \item[(i)] \textit{Separation.} The function $S^\iota$ represents the separation cost, i.e., the propension of agents to remain close to other agents of the same population and distant from those of the other population. Following \cite{abc}, we assume that 
    \begin{align*}    
    S^\iota(m_1, m_2) &= G^\iota\left(\int_{\Gamma} K(x, y) dm_1(y), \int_{\Gamma} K(x, y) dm_2(y), a^\iota, \epsilon\right),
    \end{align*}
    where
    \[
    G^\iota(r,s,\epsilon,a^\iota) = \left(\frac{r}{r+s+\epsilon} - a^\iota\right)^-, \quad
    K(x, y) = \frac{1}{2\delta}\chi_{B_\delta(x)}.\]
    Here $(a)^-$ denotes the negative part of $a \in \mathbb{R}$, $\chi_A$ is the characteristic function of the set $A$, $B_\delta(x) = \{y \in \Gamma : d_\Gamma(x, y) < \delta\}$, and $\epsilon$ is a regularization constant. The parameter $a^\iota \in (0,1)$ is a threshold of happiness: if the percentage of population $\iota$ living in $B_\delta(x)$ is below $a^\iota$, the players of this population pay a positive cost. In our model, workers prefer to live in residential areas far from factories, while firms have no particular preference for being close to other firms, hence we assume that $a^1$ is much larger than $a^2$.
    \item[(ii)] \textit{Overcrowding.} The function $O^\iota$ describes the agents' aversion to overcrowding. As in \cite{abc}, we consider
    $$
    O^\iota[m_1, m_2](x) = C^\iota \left(\int_{\Gamma} K(x, y) \frac{m_1(y) + m_2(y)}{2} dy - b^\iota\right)^+.
    $$
    An agent at $x$ starts to pay a positive cost when the average distribution of agents from both populations in the neighborhood $B_\delta(x)$ exceeds a threshold parameter $b^\iota$. Since workers are more sensitive to overcrowding, we assume that $b^2 > b^1$ and $C^2 > C^1$. It also represents the fact that, as in \cite{barilla}, the housing cost increases with the total density of the two populations.  
\end{itemize}

\begin{remark}
The function $R^\iota$, comprising the components $S^\iota$ and $O^\iota$ as previously defined, satisfies assumption \eqref{ip_coupling} if we replace $K(x,y)$ with a non-negative smooth kernel that equals 1 for $y \in B_\delta(x)$ and 0 outside a small neighborhood $B_\delta(x)$ (see \cite{abc}). This adjustment ensures that the kernel retains its essential properties while becoming smooth enough to satisfy the required assumptions.
It is possible, as in \cite{barilla}, to consider an overcrowding cost of the form $F(m_1, m_2) = f(m_1 + m_2)$, where $f$ is an increasing function that only takes into account the density of the two populations at a given position.

\end{remark}

\begin{table}[H]
    \centering
  
    \begin{tabular}{|c|c|c|l|}
        \hline
        \textbf{Parameter} & \textbf{workers}  & \textbf{firms} & \textbf{Description} \\
        \hline
        & & & \\
        $\rho^\iota$ & 1 & 4 & Mobility cost, typically lower values for workers \\ 
         & & &  and  higher values for firms ($\rho_1 < \rho_2$). \\
        $\Lambda^\iota$ & 1 & 1 & Housing cost, which depends on area-specific   \\
        & & &  factors such as connectivity and infrastructure. \\
        $a^\iota$ & 0.5 & 0.1 & Threshold of happiness, indicating the percentage   \\
        & & &   of $\iota-$population in $B_\delta(x)$ to avoid a positive cost. \\
        $b^\iota$ & 4 & 8 & Threshold parameter for overcrowding, deter-  \\
        & & & mining when  agents start  to pay a positive cost. \\
        $C^\iota$ & 1 & 1 & Coefficient for overcrowding cost\\
         $\delta$ & 0.1 & 0.1 & Radius defining the neighborhood $B_\delta(x)$ for \\
        & & &  interaction considerations.\\
        $\epsilon$ & $10^{-5}$  & $10^{-5}$  & Regularization constant used in the separation \\
        & & & cost function. \\
        $\mu^\iota_\alpha$   & 0.4         & 0.2       & Viscosity coefficient of a specific population \\
           & & & of agents  \\
        $\sigma$         & 0.5         & 0.5       & Entropic regularization introduced in the optimal\\
        & & &  transport problem \\
        $T$              & 2           & 2         & Time horizon\\
        $c$ & linear & linear & Transfer cost between two points of the network.\\
        & & & \\
        \hline
    \end{tabular}
    \caption{Model parameters and their descriptions.}
    \label{tab:my-table}
\end{table}

\begin{remark}

When not stated otherwise, we assume the model parameters as reported in Table \ref{tab:my-table}. This table also provides a summary of the  meaning of each parameter.  These parameters are essential for accurately simulating the processes of aggregation and segregation between the  populations in our model since they define the mobility, housing costs, happiness thresholds, and overcrowding sensitivities, among other factors. By carefully choosing these values, we can capture the complex dynamics and interactions within the populations under study.

\end{remark}

\subsubsection{Test 1}

In this initial test, we aim to demonstrate the various possible configurations achievable with our model. To keep our network as simple as possible, we consider a three-vertex network connecting the points in a 2D plane \(v_1 = (1,0)\), \(v_2 = (\cos(2\pi/3), \sin(2\pi/3))\), and \(v_3 = (\cos(4\pi/3), \sin(4\pi/3))\). This network is isomorphic to a periodic 1D domain, similar to the case discussed in \cite{barilla}. This simplified setup allows us to focus on the fundamental interactions and behaviors of the model without the added complexity of a more intricate network structure.

All the approximations presented in this section are obtained by implementing the schemes described above, using discretization parameters \(\Delta x = \Delta t = 0.01\). It is important to emphatize that this implementation with such a large CFL number is feasible due to the unconditional stability of the semi-Lagrangian scheme. This stability is crucial because it allows us to manage the numerical complexity of the model effectively. By leveraging the inherent stability of these schemes, we can achieve accurate and reliable results without compromising computational efficiency.

\begin{figure}[t]
\begin{center}
\begin{tabular}{c}

\includegraphics[width=0.95\textwidth]{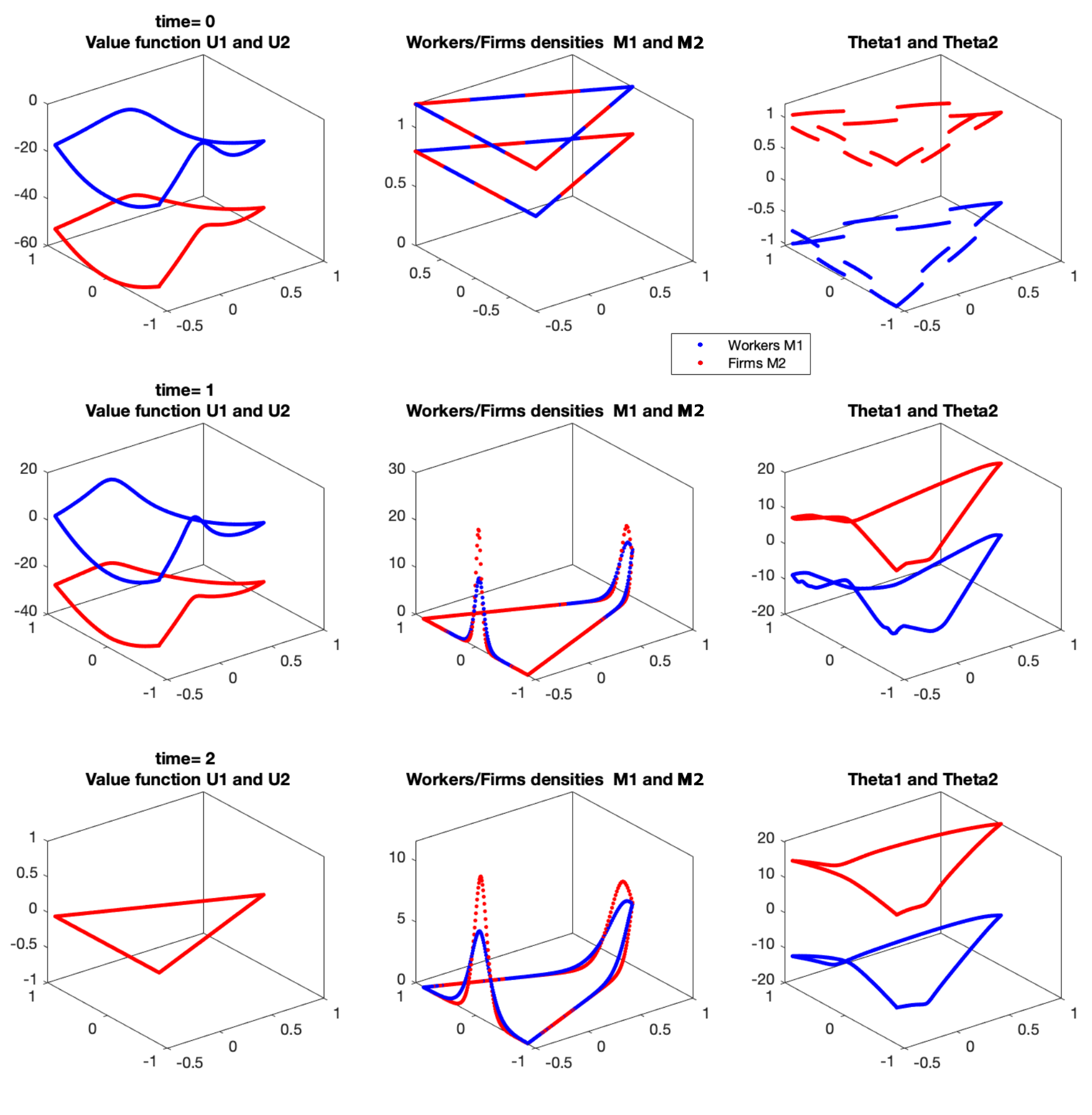}
\end{tabular}
\end{center}
\caption{Test 1a. Aggregation. Model parameters set as Table \ref{tab:my-table}.}\label{F:1}
\end{figure}

The equilibrium between the forward evolution (Fokker-Planck equation), the backward evolution (Hamilton-Jacobi equation), and the optimal transport problem is found using a fixed-point procedure. Although we lack a theoretical guarantee or an estimate of the convergence rate, we observed, as in the case of classic Mean-Field Game \cite{carlini2016DGA,CS17}, that convergence to a stable equilibrium is typically achieved within a few iterations (usually around 15-17 iterations) under a desired tolerance, which in our case is set to \(10^{-4}\). This empirical observation provides confidence in the robustness and efficiency of our approach.

In Figure \ref{F:1}, we observe that, with the parameter choices listed in Table \ref{tab:my-table}, our model produces an aggregation of the two populations (firms and workers). Specifically, starting from an initial configuration where firms and workers are distributed throughout the network with density values of 0.8 and 1.2 respectively, we observe that the populations tend to self-organize, forming two clusters on opposite sides of the network. In this case, firms and workers tend to aggregate in these two clusters to minimize the commuting cost, facilitated by a low coefficient for the overcrowding cost \(C^\iota\). Clearly, in the final configuration, the firms have a tendency to aggregate more, due to a lower viscosity coefficient \(\mu^2\).

\begin{figure}[t]
\begin{center}
\begin{tabular}{c}

\includegraphics[width=0.95\textwidth]{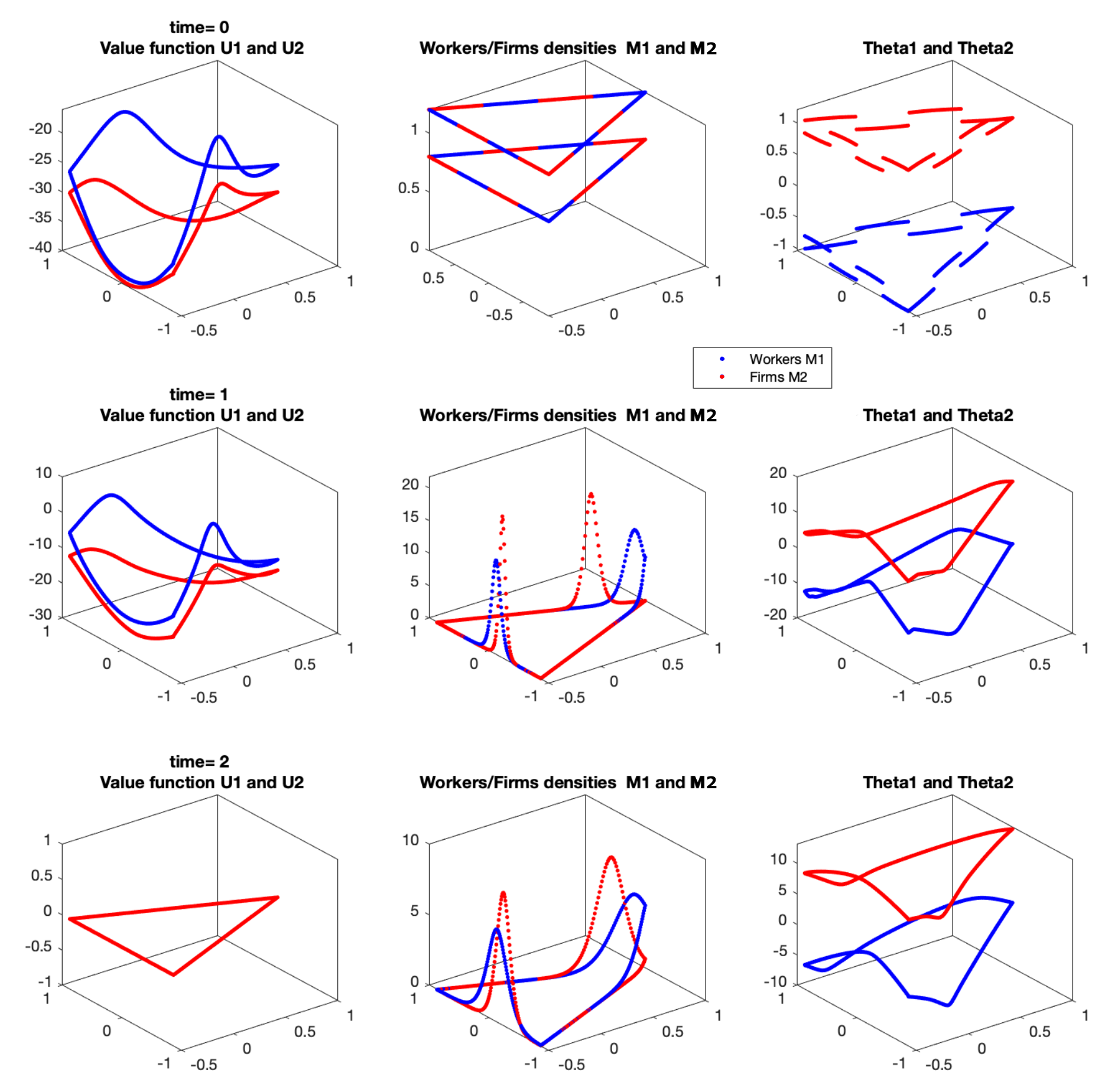}
\end{tabular}
\end{center}
\caption{Test 1b. Coexistence. Model parameters set as Table \ref{tab:my-table} with the exception of  \(C^1 = 5\) and \(C^2 = 1\).}\label{F:2}
\end{figure}

In the next scenario, altering the overcrowding cost to \(C^1 = 5\) and \(C^2 = 1\), we obtain the situation shown in Figure \ref{F:2}. Starting from the same initial distribution, we observe a slightly different final organization: firms and workers maintain a short distance from each other but only partially share the same spaces. This indicates a balance between the desire to stay close to one's population and the need to avoid overcrowding.

Finally, keeping the same overcrowding cost but setting the threshold parameter for overcrowding to \(b^1 = 1\) and \(b^2 = 2\), as shown in Figure \ref{F:3}, we achieve complete segregation of the two populations. From the same initial configuration, workers and firms migrate to opposite sides of the network, neglecting the commuting cost arising from the optimal transport problem. This drastic change in behavior highlights the sensitivity of the model to parameter variations and underscores the importance of careful parameter selection in simulating realistic scenarios.

\begin{figure}[t]
\begin{center}
\begin{tabular}{c}

\includegraphics[width=0.95\textwidth]{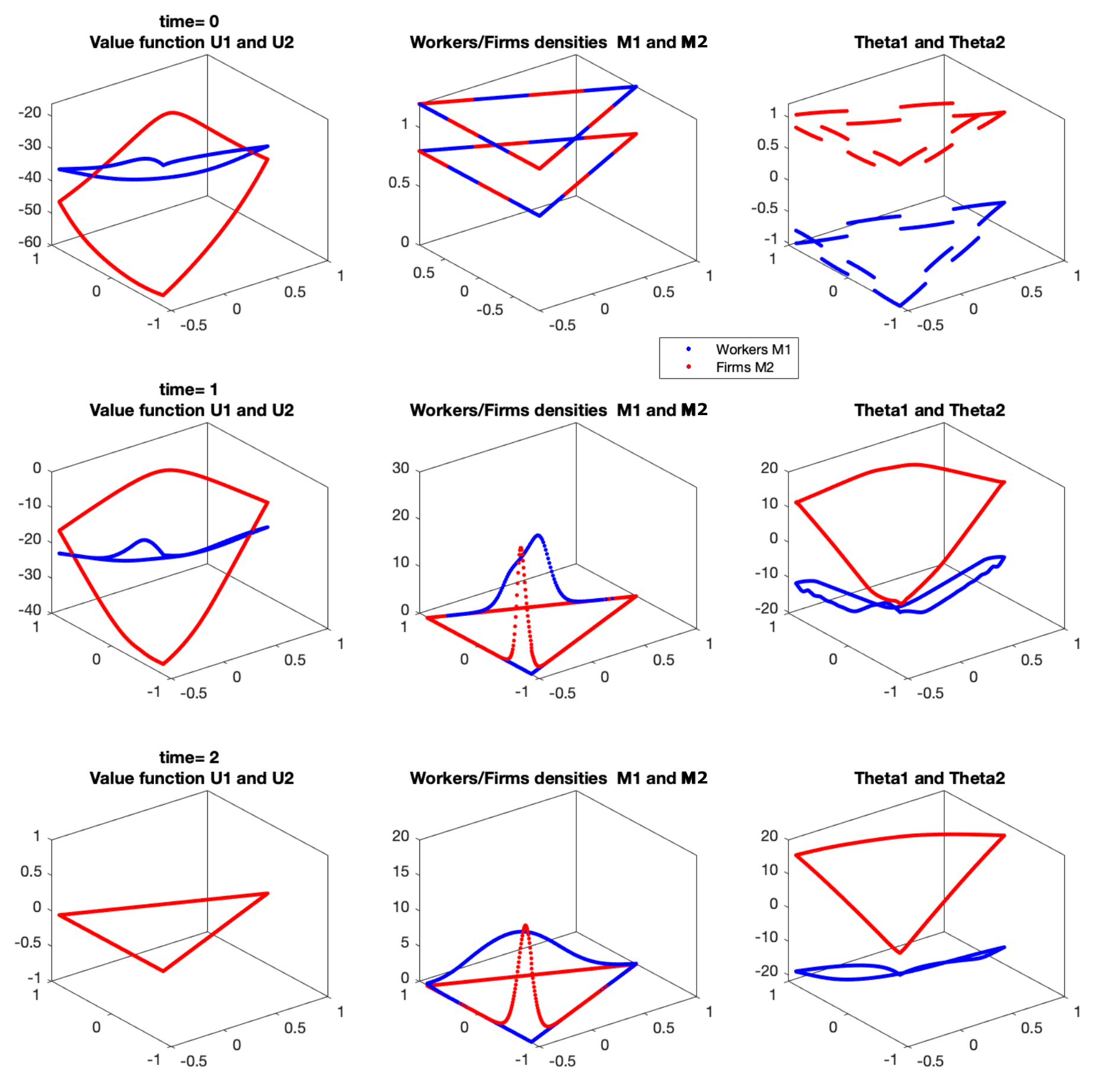}
\end{tabular}
\end{center}
\caption{Test 1c. Segregation. Model parameters set as Table \ref{tab:my-table} with the exception of  \(b^1 = 1\) and \(b^2 = 2\). }\label{F:3}
\end{figure}

It is important to keep in mind that the equilibrium solution in our model may not be unique. Consequently, starting from different initial conditions, our system may evolve to distinct final configurations. Nonetheless, our observations indicate that if the parameters are chosen to induce aggregation, coexistence, or segregation, the resulting configuration will exhibit the same qualitative features, regardless of the initial distribution.

This phenomenon is illustrated in Figure \ref{F:4}. In this case, we used the same parameters as in Figure \ref{F:3}, but we began with a different initial condition: an equal, uniform distribution of density for each population across the entire network. Despite this change, the system still evolved into a segregated configuration. However, unlike the scenario depicted in Figure \ref{F:3}, where workers and firms each formed a single cluster, in this new configuration, workers and firms each form three distinct clusters. These clusters subdivide the available space of the network between them.

This observation underscores an essential aspect of our model: while the exact spatial distribution of the populations can vary depending on the initial conditions, the overall pattern of aggregation, coexistence, or segregation dictated by the parameter choices remains consistent. This robustness highlights the model's ability to capture the fundamental dynamics of population distribution and interaction under different initial conditions.
 
\begin{figure}[h]
\begin{center}
\begin{tabular}{c}

\includegraphics[width=0.95\textwidth]{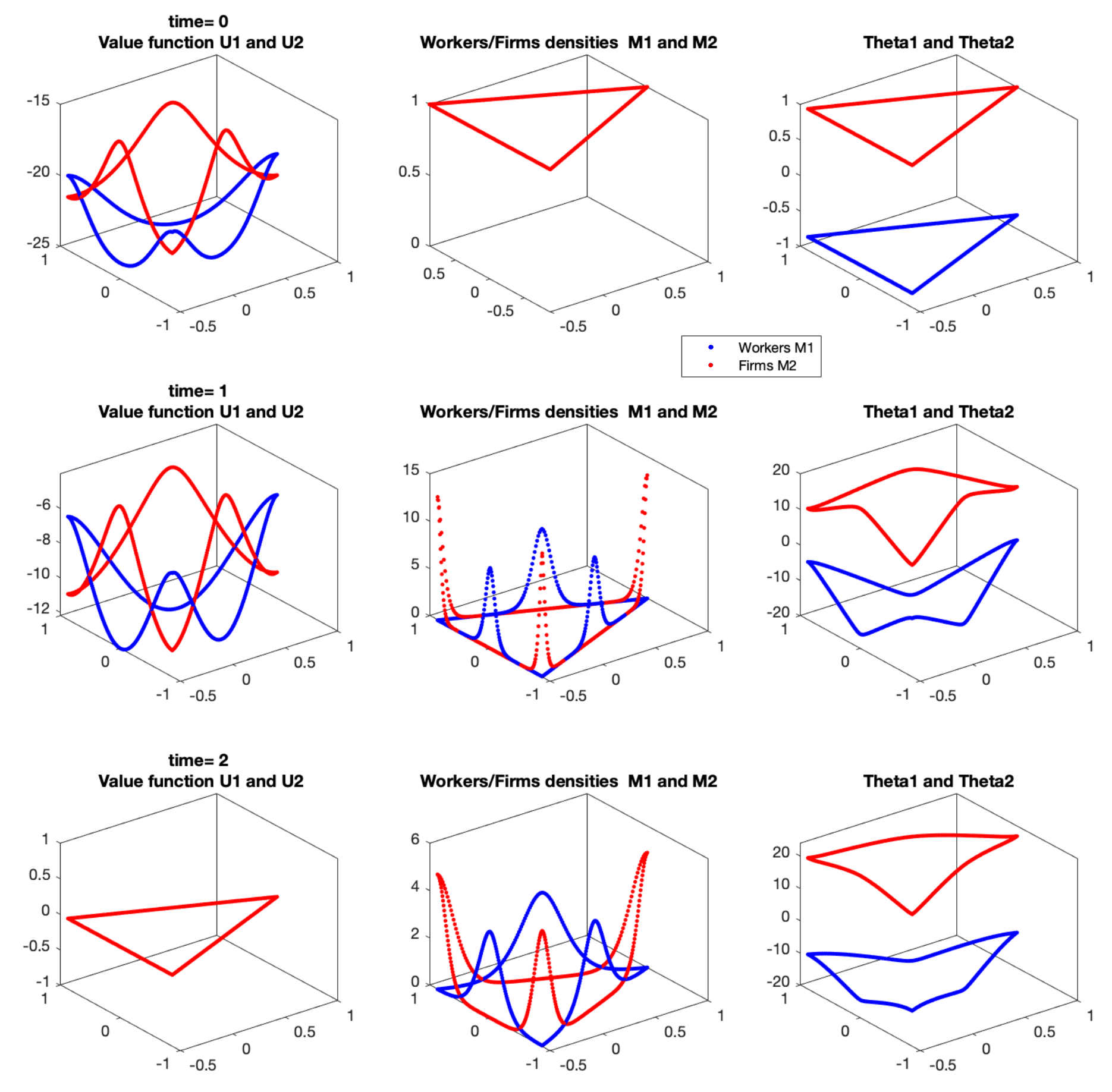}
\end{tabular}
\end{center}
\caption{Test 1d. Segregation from a different initial distribution. Model parameters set as Table \ref{tab:my-table} with the exception of  \(b^1 = 1\) and \(b^2 = 2\). }\label{F:4}
\end{figure}

In summary, while the precise equilibrium solution may vary with different initial conditions, the qualitative behavior of the system remains robust, demonstrating either aggregation, coexistence, or segregation based on the chosen parameters. This property is crucial for understanding the range of possible dynamics in population models and reinforces the significance of parameter selection in shaping these dynamics.

\subsubsection{Test 2}

In this second test, we consider a slightly more complex network where the vertices are located at \(v_1 = (0,0)\), \(v_2 = (1,0)\), \(v_3 = (1,1)\), and \(v_4 = (0,1)\), connected by five edges. Notably, vertices \(v_1\) and \(v_3\) have three connections each, while the other vertices have only two connections. This setup allows us to examine the impact of the optimal transport problem on our model. We keep all parameters constant as shown in Table \ref{tab:my-table}, except for the coefficients of the overcrowding cost, set to \(C^1 = 5\) and \(C^2 = 3\). The primary variable in this test is the transfer cost \(c(x,y)\), which we impose as a cost to minimize in the optimal transport problem. Specifically, we compare three commuting costs: linear, square root, and quadratic, defined as follows:

$$
\left\{
\begin{array}{ll}
c_{lin}(x,y) = |x-y|, & \\
c_{sqr}(x,y) = \sqrt{|x-y|}, & \hbox{ for any } x,y \in \Gamma, \\
c_{quad}(x,y) = |x-y|^2, & \\
\end{array}
\right.
$$

\begin{figure}[h]
\begin{center}
\begin{tabular}{c}
\includegraphics[width=0.85\textwidth]{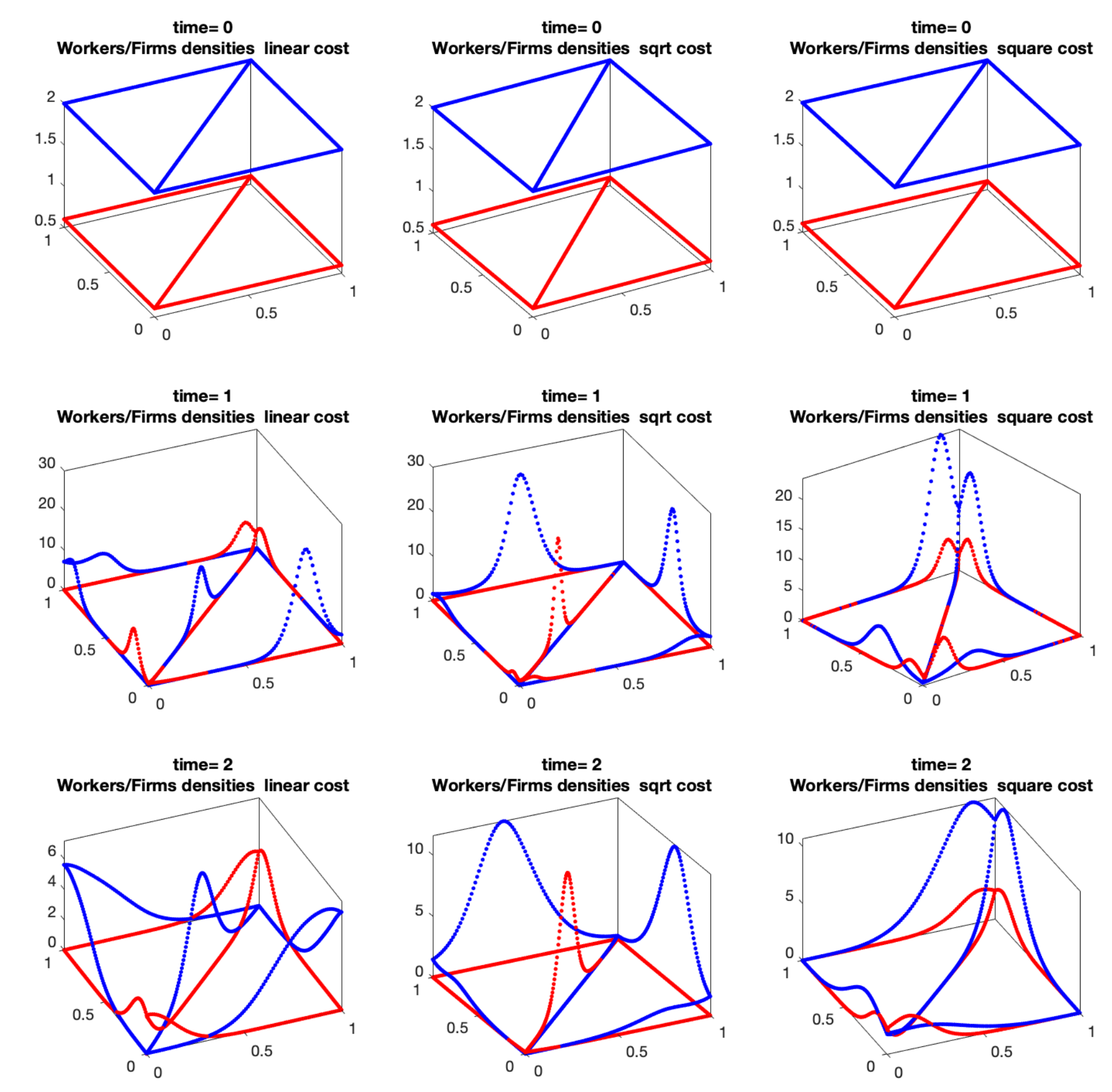}
\end{tabular}
\end{center}
\caption{Test 2. Different cost choices (linear, square root, quadratic) in the optimal transport problem. Other model parameters set as Table \ref{tab:my-table}.}\label{F:5}
\end{figure}

We present the results of this test in Figure \ref{F:5}, starting from an initial configuration where workers and firms are equally distributed across the network, but with a locally higher density for the workers' population. The figure compares the evolution of densities under the three different commuting cost functions. We observe that a quadratic cost encourages the two populations to aggregate due to the high commuting cost, promoting a clustering effect to minimize movement expenses. Conversely, a square root cost, representing a lower commuting cost, incentivizes the segregation of the two populations. This results in more dispersed configurations where populations occupy distinct regions to reduce overlap.

An intermediate scenario is observed with a linear cost, where the firms' population tends to settle in the central arc of the network, while workers aggregate into two symmetric clusters along the outer edges. This pattern is reminiscent of the typical configurations found in large, rapidly developing cities, where commercial enterprises often occupy central locations within the road network, while residential areas, not too distant from the city center, are preferred by the working population.

These results highlight the significant impact that different commuting cost functions can have on the spatial distribution of populations. By varying the cost function, we can simulate diverse urban planning scenarios, providing insights into how different transport policies might influence the organization of urban areas.

In summary, this test demonstrates that the choice of commuting cost function in the optimal transport problem can lead to significantly different equilibrium configurations. Quadratic costs tend to promote aggregation, linear costs result in a mixed structure with central clustering of firms, and square root costs foster segregation. This highlights the model's flexibility and its potential applications in urban planning and policy-making.

\subsubsection{Test 3 - A Development Simulation for New Paris Areas Post-Olympics}

The recent Olympic Games of Paris 2024 have brought significant investments and a comprehensive reorganization of the area between the three suburbs of Saint-Denis, Saint Ouen, and L'Île-Saint-Denis (see Figure \ref{F:5b}). This redevelopment focuses on the three villages constructed for the Paris Games, which are strategically spread across the northern suburbs of Paris, all within the Seine-Saint-Denis department. These villages will accommodate 4,250 athletes during the Olympics and 8,000 athletes during the Paralympics.

Considerable effort has been dedicated to making these facilities sustainable by repurposing former industrial buildings into amenities and accommodations for the athletes. This urban regeneration initiative is expected to provide new opportunities for residents and businesses in the area once the Games conclude.

\begin{figure}[h]
\begin{center}
\begin{tabular}{c}
\hspace{-1cm}
\includegraphics[width=0.45\textwidth]{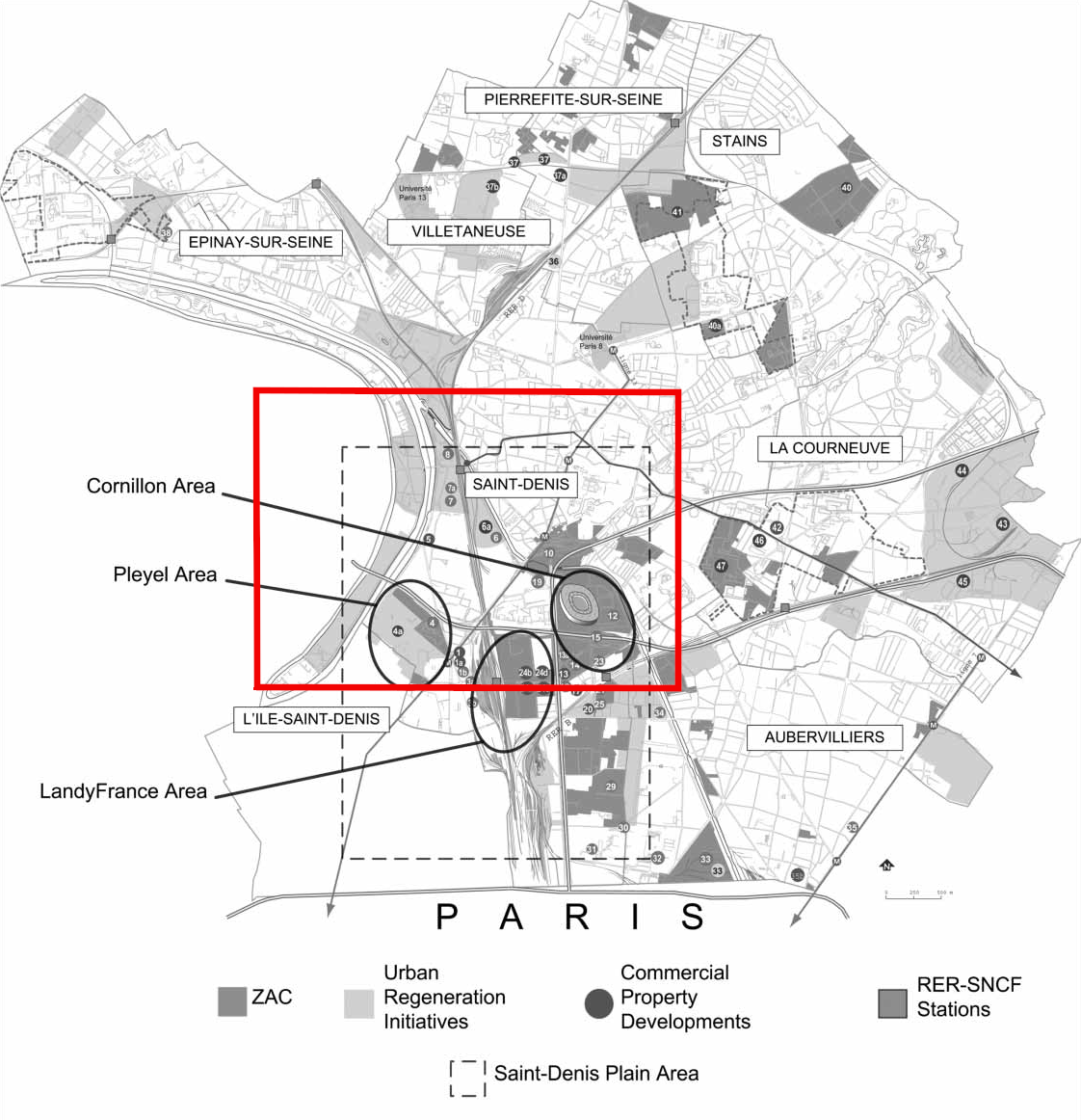}
\includegraphics[width=0.56\textwidth]{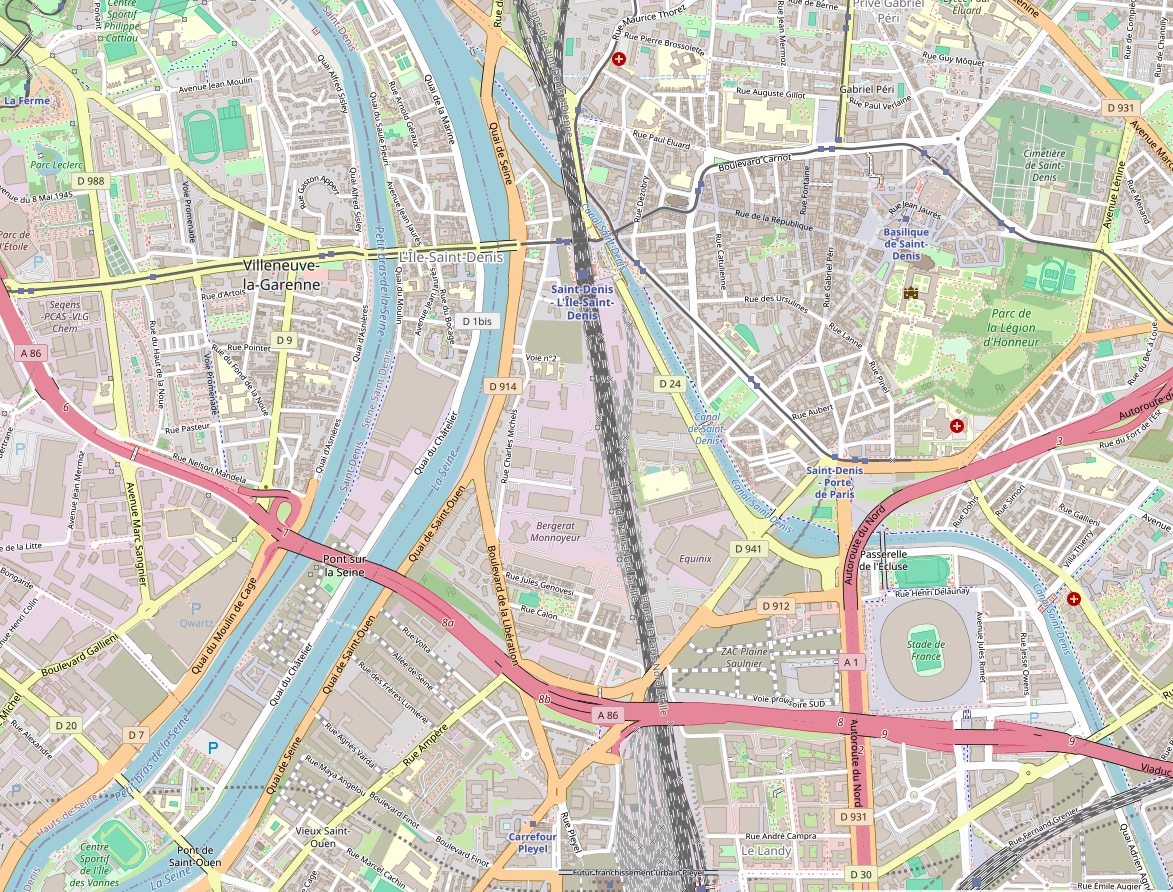}
\end{tabular}
\end{center}
\caption{Test 3. The Saint Denis plane in the \emph{grand projet de ville} in the northern part of the Paris region  \cite{NC06}. A close-up of the area around the Olympic Village. }\label{F:5b}
\end{figure}

In particular, the urban area of Saint-Denis (see Area 1 in Figure \ref{F:6}, left) is a densely populated suburb with a growing population of approximately 115,000 individuals. Additionally, the area around the Connection Hub Saint-Denis--Pleyel (see Area 2 in Figure \ref{F:6}, left) hosts a significant number of companies and is projected to employ between 15,000 and 22,000 people within a one-kilometer radius of the station.

These developments underscore the transformative impact of the Olympic Games on urban infrastructure and local economies. By revitalizing industrial spaces and enhancing connectivity, the initiative aims to foster sustainable growth and improve the quality of life for local residents. The strategic positioning of the athlete villages and the focus on sustainability highlight the Games' legacy, ensuring long-term benefits for the Seine-Saint-Denis department and its inhabitants.

\begin{figure}[h]
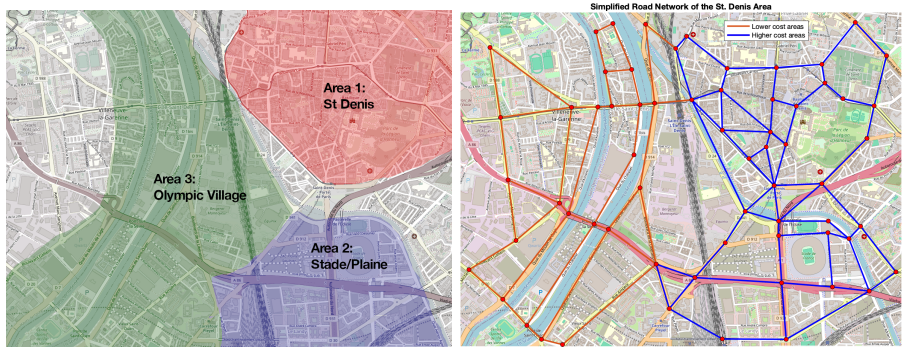

\begin{center}
\begin{tabular}{c}
\hspace{-0.5cm}
\includegraphics[width=0.48\textwidth]{Figures/parisOlympicsAreas.png}
\includegraphics[width=0.48\textwidth]{Figures/ParisNet.png}
\end{tabular}
\end{center}
\caption{Test 3. A subdivision of the area of interest and the road network that we consider (left) and the simplified road network (right). The blue lines indicate roads in areas with limited accessibility, while the orange lines highlight areas that underwent significant renovations during the 2024 Olympics.}\label{F:6}
\end{figure}

We use this example to test our model in a more realistic case. However, the purpose is mostly illustrative, as tuning the model parameters and conducting a data-driven experimental evaluation of the results are outside the objectives of this paper.

We consider a simplified road network, as shown in Figure \ref{F:6}, right. The network consists of 64 vertices and 101 roads, with multiple connections at each vertex, up to 6. We consider the network isolated from the rest of the city, therefore, no boundary conditions are imposed in the model.

\begin{figure}[h]
\begin{center}
\begin{tabular}{c}
\hspace{-0.8cm}
\includegraphics[width=1\textwidth]{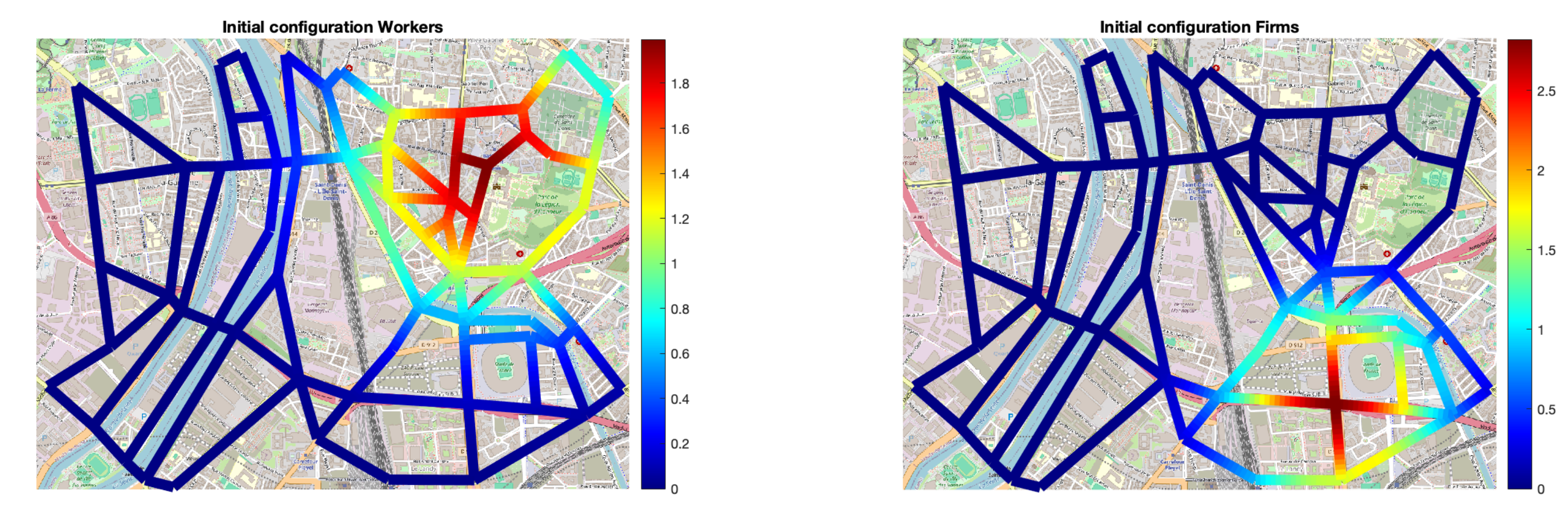}
\end{tabular}
\end{center}
\caption{Test 3. Initial conditions of our test. The workers are concentrated around the center of the Saint-Denis area, while the firms (more concentrated) are initially located in the Stade/Plaine area.}
\end{figure}

As initial conditions, we concentrate the workers predominantly around the center of Saint-Denis, while the firms are highly concentrated around the Stade de France and its surrounding facilities. Specifically, given that the network lies within the 2D box $[0,12]^2$, the initial distributions are defined as:

$$
m_0^\iota(x,y) = \frac{1}{\gamma\pi} e^{-\frac{(x-v_1^\iota)^2+(y-v_2^\iota)^2}{\gamma}},
$$
with \(v_1 = (8.94, 9.22)\), \(\gamma = 8\), and \(v_2 = (8.55, 4.5)\), \(\gamma = 4\). The housing cost \(\Lambda^\iota_\alpha\) is set to 5 on historically urbanized roads (marked in blue in Figure \ref{F:6}) and 1 on roads affected by the Olympic Games renovations (in orange in Figure \ref{F:6}).

The remaining parameters of the model are as listed in Table \ref{tab:my-table}, with the exceptions of \(b^1 = 1\) and \(b^2 = 2\). The transfer cost used in the Optimal Transport problem is linear with respect to the distance. The discretization parameters are \(\Delta x = 0.1\) and \(\Delta t = 0.05\).

\begin{figure}[h]
\begin{center}
\begin{tabular}{c}
\hspace{-0.6cm}
\includegraphics[width=1\textwidth]{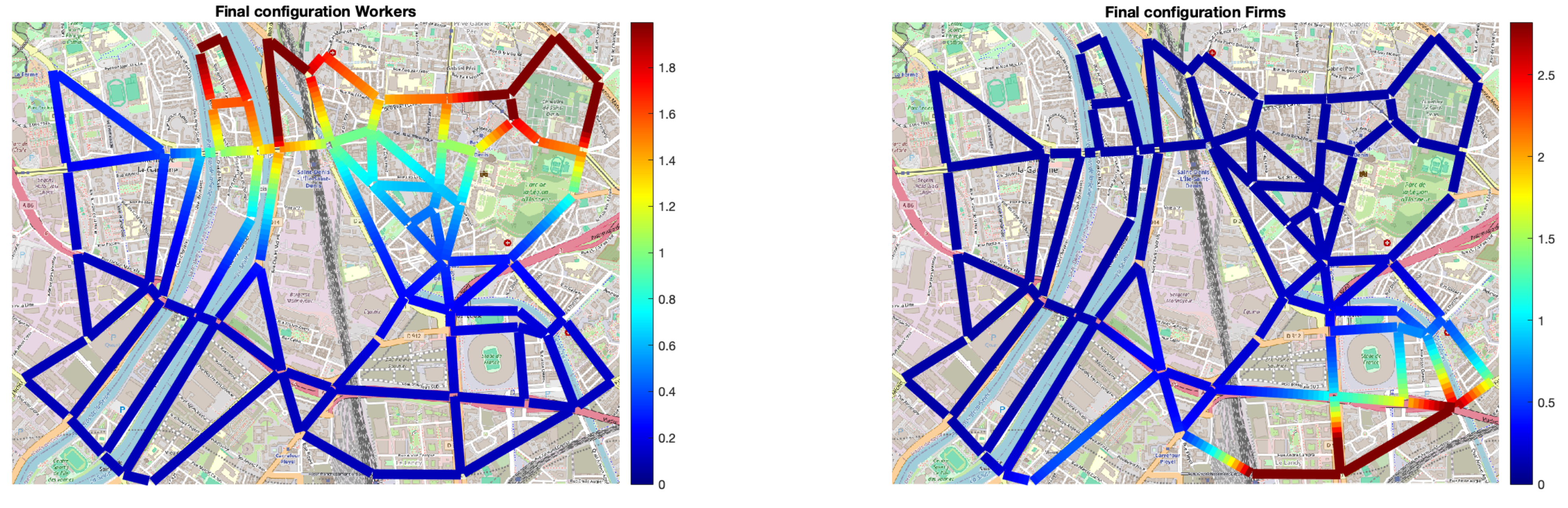}
\end{tabular}
\end{center}
\caption{Test 3. Final configuration at \(T=2\). The workers form two clusters in the northern part of the network, while the firms concentrate in the southern area near Stade/Plaine.}\label{F:8}
\end{figure}

We observe the evolution of the two populations until reaching a final configuration at \(T = 2\), as shown in Figure \ref{F:8}. Given the segregation between the two populations observed in Test 1 with these parameters, we anticipated a similar outcome. Workers and firms distribute themselves on distinct sides of the network. However, unlike Test 1c, they do not occupy opposite areas of the network, which can be attributed to the commuting cost incurred for travel between workers' and firms' locations. Additionally, the non-constant housing cost introduces an intriguing dynamic: workers split into two clusters, one in the high-cost northern area of Saint-Denis, and another in the lower-cost new Olympic Village area. This distribution balances housing and commuting costs, as firms remain on the right side of the network. Interestingly, the new renovated areas remain underutilized by both populations, likely due to high mobility costs and the inherent network geometry.

{
While the present version of the model considers homogeneous agents within each population, 
the framework can readily accommodate \emph{heterogeneous preferences} or \emph{production types}, 
for instance by introducing dependence of the Lagrangian or the coupling costs on additional agent-specific parameters. 
This extension would allow the analysis of \emph{income-based segregation} or \emph{firm specialization} effects, 
and will be the subject of future research. 

Moreover, the parameters employed in the numerical experiments have been chosen to illustrate the qualitative behavior of the system, 
but the model can also be \emph{calibrated to empirical data}--for example, by matching observed commuting flows or density distributions. 
Such an empirical calibration would enable quantitative applications to real urban systems 
and strengthen the connection between theoretical modeling and observed spatial patterns.}

\section{Conclusions}

This study provides a comprehensive analysis of a mathematical model designed to simulate urban evolution, driven by interactions between two large populations: workers and firms. Utilizing a network to represent the city's map, with edges depicting residential areas and communication routes, the model offers a robust framework to understand the dynamics of urban planning through a two-population Mean-Field Game system coupled with an Optimal Transport problem.

The findings establish the existence and uniqueness of solutions, derived through a fixed-point argument, circumventing the need for symmetric interactions between the populations. This approach not only validates the model's theoretical foundation but also highlights the realistic assumptions about human preferences and industrial clustering. People tend to avoid living near polluting factories or in overcrowded areas, while industries gravitate towards clustering for better transportation efficiency.

Moreover, the numerical simulations provide practical insights into urban planning. They reveal how varying parameters, such as the commuting costs represented by edge lengths, significantly impact urban development. The simulations also demonstrate the model's capability to simulate urban patterns under different scenarios, proving its possible utility for policymakers and urban planners.

{From an economic perspective, the equilibria produced by our model illustrate mechanisms of segregation and clustering akin to those discussed in the urban economics literature. High commuting or mobility costs lead to spatial segregation between workers and firms, while stronger network connectivity and lower relocation costs promote mixed equilibria resembling polycentric urban forms. Differences in the network topology correspond to alternative urban infrastructures, highlighting how transport investments or zoning policies can alter the balance between agglomeration and dispersion. Thus, beyond its mathematical structure, the model provides a stylized yet flexible tool for analyzing policy interventions such as congestion pricing or the expansion of transport capacity within a unified Mean-Field Game--Optimal Transport framework.}

In conclusion, this research bridges the gap between theoretical urban planning models and real-world applications. It underscores the importance of considering both human preferences and industrial needs in urban development. The model's versatility in simulating various urban scenarios offers a valuable tool for designing sustainable and efficient cities, aligning with contemporary urbanization trends and challenges. Future work could expand on this model by integrating additional factors such as environmental impacts and economic policies, further enhancing its applicability and accuracy in urban planning. Additionally, incorporating Dirichlet boundary conditions, where agents can enter or exit the system, would allow for changes in the number of workers and firms, enabling the city to evolve into different configurations.
\par\medskip\noindent
{\bf Data availability} The present manuscript has no associated data.

\begin{acknowledgements}
This work has been partially supported by the ``INdAM-GNAMPA Project'' (CUP E53C23001670001) ``Modelli mean-field games per lo studio dell'in\-qui\-na\-mento atmosferico e i suoi impatti'' and by the Gruppo Nazionale per l'Analisi Matematica e le loro Applicazioni (GNAMPA-INdAM), INdAM-GNAMPA projects 2022 and 2023, Gruppo Nazionale per il Calcolo Scientifico’ (GNCS-INdAM) and  by PRIN project 2022 (Funds 2022238YY5) “Optimal control problems: analysis, approximation ”.
\end{acknowledgements}

\appendix
\section{Appendix}
\label{app:A}

\subsection{Results for the Optimal Transport problem}
For $m_1$, $m_2\in\cM$, consider the Optimal Transport problem
\begin{equation}\label{OT}
	 \cC(m_1, m_2)=\inf_{\gamma\in\Pi(m_1, m_2)}\int_{\Gamma\times\Gamma}c(x, y)d\gamma(x, y),
\end{equation}
where the cost $c$  satisfies \eqref{comm_cost} and its dual formulation
$$ \cC(m_1, m_2)=\sup_{\phi, \psi\in C(\Gamma)}\left\{\int_{\Gamma}\phi dm_1+\int_{\Gamma}\psi dm_2:\phi(x)+\psi(y)\leq c(x, y)\ \text{for every }x,y\in\Gamma\right\}.	$$
Observe that if $\phi$ is a Kantorovich potential for \eqref{OT}, then the couple $(\phi-\alpha, \phi^c+\alpha)$ is still a Kantorovich potential for any $\a\in\R$.
The results that we report below are well known, even in much more general contexts \cite{santambrogio,villani2}, and  we briefly sketch the adaptation to the network setting.
\begin{prop}\label{normkantorovich}
	Let $\phi$ be a Kantorovich potential for \eqref{OT} such that $\int_{\Gamma}\phi(x)dx=0$. Then $\phi$ is Lipschitz continuous on $\G$ and $\|\phi\|_{L^\infty(\G)}, \|\phi^c\|_{L^\infty(\G)} \le C$ with $C$ depending only on $c$ and $\Gamma$.
\end{prop}
\begin{proof}
Since \eqref{comm_cost} holds, i.e. $c\in C^1(\Gamma\times\Gamma)$, $c$ is Lipschitz continuous as well, that is there exists $L>0$ such that
$$
|c(x, y)-c(x', y')|\leq L(d_{\Gamma}(x, x')+d_{\Gamma}(y, y'))\quad\text{for every }x,x',y,y'\in\Gamma.
$$
Moreover since $\phi$ is a $c$-concave function, for every $x\in\Gamma$ it can be written as $\phi(x)=\inf_{y\in\Gamma}g_y(x)$ with $g_y(x):=c(x, y)-\phi^c(y)$ and the functions $g_y$ satisfy
$|g_y(x)-g_y(x')|\leq Ld_{\Gamma}(x, x')$, i.e. they are equi-Lipschitz continuous. This proves that $\phi$ is Lipschitz continuous (being the infimum of a family of equi-Lipschitz functions) and shares the same Lipschitz constant of $c$. Furthermore, we normalize $\phi$ and $\phi^c$ as
$$
\tilde\phi(x)=\phi(x)- \int_{\Gamma}\phi(y)dy,\quad\tilde\phi^c(x)=\phi^c(x)+\int_{\Gamma}\phi(y)dy.
$$
Hence, by \eqref{eq:int}, $\int_\Gamma\tilde\phi(x)dx=0$ and, by the Lipschitz continuity of $\phi$ and the boundedness of $c$, we have a uniform bound on $\tilde\phi$ and $\tilde\phi^c$. Indeed, by \eqref{eq:int}, we have
\begin{align*}	
 |\tilde\phi(x)|&=\left|\phi(x)-   \int_{\Gamma}\phi(y)dy\right|\leq \int_{\Gamma}|\phi(x)-\phi(y)|dy\\
&\leq   L\int_{\Gamma}d_{\Gamma}(x, y)dy\leq L\operatorname{diam}(\Gamma),
\end{align*}
where $\operatorname{diam}(\Gamma)= \sup_{x,y\in \Gamma}d_\Gamma (x,y)$, and
\begin{multline*}
|\tilde\phi^c(x)|=\inf_{y\in\Gamma}\{c(x, y)-\phi(y)\}+ \int_{\Gamma}\phi(y)dy\\
=\inf_{y\in\Gamma}\left\{c(x, y)-\left(\phi(y)- \int_{\Gamma}\phi(y)dy\right)\right\}=\inf_{y\in\Gamma}\{c(x, y)-\tilde\phi(y)\},
\end{multline*}
which implies
$$
|\tilde\phi^c(x)|\leq\|c\|_{L^{\infty}(\Gamma\times\Gamma)}+\|\tilde\phi\|_{L^{\infty}(\Gamma)}.
$$
\end{proof}
We recall that the support of a measure $m \in\cM$, denoted with $\supp(m)$, is the complement of the largest open set which is negligible for that measure.
\begin{prop}	\label{optimaltransportnetwork1}
Let $\gamma$  and  $\phi$  be an optimal   plan  and, respectively, a Kantorovich potential for \eqref{OT}. Then 
	  \begin{itemize}
	  	\item[$(i)$] For all $x, y\in\supp(\gamma)$
	  	\begin{equation}\label{kant-rub2}
	  		\phi(x)+\phi^c(y)=c(x, y) .
	  	\end{equation}
	  	\item [$(ii)$] If  $x_0, y_0\in \supp(\gamma)$, $x_0\not\in \cV$ and 
	  	$\phi$ is differentiable at $x_0$, then $\partial\phi(x_0)=\partial_xc(x_0, y_0)$.
	  \end{itemize} 
\end{prop}
\begin{proof}
Recall that, by the definition of Kantorovich potential, we have
	\begin{equation*}
 0\le	\phi(x)+\phi^c(y)\leq c(x, y),\quad\text{for all }(x, y)\in\Gamma\times\Gamma,
	\end{equation*}
	and
	$$
	\int_{\Gamma\times\Gamma}c(x,y)d\gamma=\int_{\Gamma}\phi(x) dm_1+\int_{\Gamma}\phi^c(y)dm_2=\int_{\Gamma\times\Gamma}(\phi(x)+\phi^c(y))d\gamma.
	$$
	It follows that \eqref{kant-rub2} holds $\gamma$-a.e. on $\G\times\Gamma$. Since $\phi$ and $\phi^c$ are continuous on $\Gamma$, the equality \eqref{kant-rub2} is satisfied on a closed set, i.e. on the support of the measure $\gamma$.
	\par
	Let   $x_0, y_0\in \supp(\gamma)$, $x_0\not\in \cV$, be such that $\phi$ and $c(\cdot, y_0)$ are differentiable at $x_0$. By \eqref{kant-rub2} and the definition of $\phi^c$, it  follows that the function
	$$
	x\longmapsto\phi(x)-c(x, y_0)
	$$
	has a minimum in $x=x_0$. Hence  we have that $\partial\phi(x_0)=\partial_xc(x_0, y_0)$.
\end{proof}
The next result is proved in the Euclidean case in \cite[Proposition 7.18]{santambrogio}.
\begin{prop}	\label{optimaltransportnetwork2}
	Assume  that at least one of the measures $m_1$, $m_2$ in \eqref{OT} is supported on the whole $\Gamma$. Then, up to additive constants, there exists  a unique Kantorovich potential for  \eqref{OT}.
\end{prop}
\begin{proof}
Let us assume that $\supp(m_1)=\Gamma$. By Proposition \ref{normkantorovich}, a Kantorovich potential is bounded and Lipschitz continuous,  therefore it is differentiable a.e. on $\G$. Consider two  Kantorovich potentials $\phi_0$ and $\phi_1$. By Proposition \ref{optimaltransportnetwork1}, it follows that $\partial\phi_0=\partial\phi_1$ on $\supp(m_1)\cap (\G\setminus\cV)$ and therefore a.e. in $\Gamma$.  Since $\Gamma$ is connected, we have that $\phi_0-\phi_1$ is constant on $\G$. If the measure with full support is $m_2$, we  apply the previous argument to get the uniqueness of $\psi=\phi^c$. Then, from $\phi=\psi^c$, we also recover the uniqueness of $\phi$.
\end{proof}
We conclude with a stability property for the Kantorovich potentials.
\begin{prop}\label{prop:stability_pot}
	Let $(m_1,m_2)$, $(m^n_1,m^n_2)\in \cM\times \cM$ be such that $\Wass_1(m^n_\iota,m_\iota)\lra 0$ for $n\to\infty$, $\iota=1,2$. Assume that $\supp(m_1)=\supp(m^n_1)=\Gamma$ and let $\phi$, $\phi^n$ be  the Kantorovich potential, renormalized in such a way that $\int_\G \phi(x)dx=\int_\G\phi^n(x)dx=0$, for $ \cC(m_1,m_2)$ and, respectively, $ \cC(m^n_1,m^n_2)$.
	Then $\phi^n\lra\phi$   uniformly in $\G$.
\end{prop} 
\begin{proof}
First observe that, by Proposition \ref{normkantorovich}, the sequence $\{\phi^n\}_n$ is uniformly bounded and  Lipschitz continuous in $\G$. Hence, by Ascoli-Arzela's Theorem, it converges, up to a subsequence, to a continuous $\tilde \phi$ on $\G$ with $\int_\G\tilde\phi dx=0$. Moreover, $\phi^{n,c}\lra\tilde\phi^c$ and $\tilde\phi$ is a $c$-concave function.
If $L$ is the Lipschitz constant of $\phi^{n}$ and $\phi^{n,c}$, see Proposition \ref{normkantorovich}, we have
\begin{align*}
	& \cC(m_1^n,m_2^n)\\
&	=\int_{\Gamma} \phi^n(x) d(m^n_1-m_1)+\int_{\Gamma}\phi^{n,c}(y)d(m^n_2-m_2)+
	\int_{\Gamma}\tilde\phi^n(x) dm_1+\int_{\Gamma}\tilde\phi^{c,n}(y)dm_2\\
&\le L(\Wass_1(m^n_1,m_1)+\Wass_1(m^n_2,m_2))+	 
\int_{\Gamma}\tilde\phi^n(x) dm_1+\int_{\Gamma}\tilde\phi^{c,n}(y)dm_2.
\end{align*}
Passing to the limit for $n\to\infty$ in the previous identity and since $ \cC(m^n_1,m^n_2)\lra  \cC(m_1,m_2)$ for $n\to \infty$  (see \cite[theorem 5.20]{villani2}), we have
\[ \cC(m_1,m_2)\le \int_{\Gamma}\tilde\phi(x) dm_1+\int_{\Gamma}\tilde\phi^c(y)dm_2 ,  \]
and therefore $\tilde \phi$ is Kantorovich potential for $ \cC(m_1,m_2)$. Hence, by Proposition  \ref{optimaltransportnetwork2}, $\tilde \phi=\phi$ and all the sequence $\phi^n$ converges to $\phi$ uniformly in $\G$.
\end{proof}


\end{document}